\documentclass[11pt]{amsart}
\usepackage{amssymb,amsmath}
\usepackage{epsfig}

\usepackage{hyperref}
\usepackage{fullpage}
\usepackage{color}

\usepackage[T1]{fontenc}
\usepackage{booktabs}
\usepackage{mathrsfs}
\usepackage{array, longtable}
\newcolumntype{C}{>{$}c<{$}}

\theoremstyle{plain}
\newtheorem{theorem}{Theorem}[section]

\newtheorem{lemma}{Lemma}[section]

\newtheorem{example}{Example}[section]

\theoremstyle{definition}

\newtheorem{remark}{Remark}[section]

\numberwithin{equation}{section}

\setbox0=\hbox{$+$}
\newdimen\plusheight
\plusheight=\ht0
\def\+{\;\lower\plusheight\hbox{$+$}\;}

\setbox0=\hbox{$-$}
\newdimen\minusheight
\minusheight=\ht0
\def\-{\;\lower\minusheight\hbox{$-$}\;}

\setbox0=\hbox{$\cdots$}
\newdimen\cdotsheight
\cdotsheight=\plusheight

\def\cds{\lower\cdotsheight\hbox{$\cdots$}}

\makeatletter
\renewcommand{\pmod}[1]{\allowbreak\mkern6mu({\operator@font mod}\,\,#1)}
\makeatother

\newcommand{\Prin}{\mathrm{Prin}}

\newcommand{\Cl}{\mathrm{Cl}}
\newcommand{\Ord}{\mathrm{Ord}}
\newcommand{\Orb}{\mathrm{Orb}}

\newcommand{\lcm}{\mathrm{lcm}}

\usepackage{tabularx}

\begin{document}
\title{Enumeration of modular forms for $\Gamma_1(N)$}
\author{Timothy Huber, Jeffery Opoku, and Dongxi Ye}

\address{
School of Mathematical and Statistical Sciences, University of Texas Rio Grande
Valley, Edinburg, Texas 78539, USA}
\email{timothy.huber@utrgv.edu} 
\email{jeffery.opoku01@utrgv.edu} 

\address{
School of AI and Liberal Arts, Beijing Normal-Hong Kong Baptist University, Zhuhai 519082, Guangdong,
People's Republic of China}
\email{dongxiye@bnbu.edu.cn}

\thanks{Dongxi Ye was supported by the Guangdong Basic and Applied Basic Research Foundation (Grant No. 2024A1515030222) and the BNBU Start-up Research Fund (Grant No. R0700157-26).}

\maketitle
\allowdisplaybreaks

\begin{abstract}
This paper considers holomorphic modular forms for $\Gamma_1(N)$ of integral weight
of the form
$$f^{(N)}_{\mathbf a}(\tau)
 =q^{s} (q^{N};q^{N})_{\infty}^{a_0}\prod_{j=1}^{\lfloor N/2 \rfloor}(q^j,q^{N-j};q^N)_\infty^{a_j}, \quad \mathbf a = (a_1, \ldots, a_{\lfloor N/2 \rfloor}),$$
 for fixed $a_0=2k \in 2 \Bbb Z_{\ge 0}$. We show that the number of relevant exponent vectors $\mathbf a$ is finite and characterize them in terms of the $\mathbb{Q}$-rational cuspidal divisor class group of $X_{1}(N)$. Effective procedures are given for counting the admissible exponents by enumerating the corresponding polytopes. This leads to formulas for the number of exponent vectors in terms of quasipolynomials in $k$.
\end{abstract}

\section{Introduction}

This work is concerned with characterizing and counting products of the form
\begin{align} \label{eq:fa-definition1}
f^{(N)}_{\mathbf a}(\tau)
=q^{\frac{N}{24} S_1(a_0,\mathbf a)\;-\;\frac{1}{2N}\,S_2(\mathbf a)} (q^{N};q^{N})_{\infty}^{a_0}\prod_{j=1}^{\lfloor N/2 \rfloor}(q^j,q^{N-j};q^N)_\infty^{a_j}, \quad q=e^{2\pi i\tau},\ {\rm Im}(\tau)>0,
\end{align}
that are holomorphic modular forms for $\Gamma_1(N)$.
Here $5 \le N \in \Bbb Z$, $a_{i} \in \Bbb Z$, and for each fixed $0 \leq a_0\in 2\Bbb Z$ and exponent vector $\mathbf a=(a_1,\dots,a_m)\in \mathbb Z^{\lfloor N/2 \rfloor}$, the exponent of $q$ is given in terms of
\begin{align} \label{S1S2}
S_1(a_0, \mathbf a)=a_0 + 2\sum_{j=1}^{m} a_j ,\qquad 
S_2(\mathbf a)=\sum_{j=1}^{m} j(N-j)\,a_j, \qquad m = \left \lfloor \frac{N}{2} \right \rfloor.
\end{align}
For $N \ge 5$ and $k\geq0$, define the set of exponents $\mathbf a$ so that $f_{\mathbf a}^{(N)}$ is modular on $\Gamma_1(N)$ of weight $k$,
\begin{equation}
\label{LNA}
L_N(k)
=
\left\{
\mathbf a =(a_1,\dots,a_m)\in \mathbb{Z}^m :
f_{\mathbf a}^{(N)}(\tau)\in M_{k}(\Gamma_1(N)),\, a_{0}=2k
\right\}, 
\end{equation}
where $M_{k}(\Gamma_1(N))$ denotes the space of holomorphic modular forms for $\Gamma_{1}(N)$ of weight~$k$. Note that $L_{N}(0)=\{(0,\ldots,0)\}$. Since an element of \(M_{k}(\Gamma_1(N))\) may be viewed as a section of the
line bundle of weight \(k\) on the modular curve \(X_1(N)\), and 
\(f_{\mathbf a}^{(N)}=f_{\mathbf a}^{(N)}(\tau)\) is holomorphic and nonvanishing on the upper half-plane
\(\mathbb H\), the behavior
of \(f_{\mathbf a}^{(N)}\) at cusps determines membership in \(L_N(k)\), and thus, \(L_N(k)\) can be characterized
in terms of modular units on \(X_1(N)\) with prescribed cuspidal divisors. The special form of \(f_{\mathbf a}^{(N)}\) implies that the relevant cuspidal
divisors are \(\mathbb Q\)-rational. Consequently, \(L_N(k)\) can be described
using the \(\mathbb Q\)-rational cuspidal divisor class group of \(X_1(N)\).
This description leads to the characterization in
Theorem~\ref{mt}, where \(L_N(k)\) is realized as a finite set of lattice
points in a bounded polyhedron satisfying explicit congruence conditions. By a result of Streng~\cite{Str}, the \(\mathbb Q\)-rational cuspidal divisor
class group of \(X_1(N)\) can be computed explicitly. We use Streng's construction to find a $\Bbb Z$-basis for the group of modular units in 
Lemma~\ref{prop:explicit-BN}. The characterization of $L_{N}(k)$ in terms of the divisor class group in Theorem \ref{mt} is then used to formulate $|L_{N}(k)|$. 

In previous work \cite{huber2024ramanujan}, we proved that 
\begin{equation}
\label{levels57}
|L_{5}(k)|=\sum_{j=0}^{\left\lfloor k/2\right\rfloor}
                 \binom{k-2j+1}{1}\quad\mbox{and}\quad |L_{7}(k)|=\sum_{j=0}^{\left\lfloor 2k/3\right\rfloor}
                 \binom{2k-3j+2}{2}.
\end{equation}
In these cases, the corresponding modular
curves \(X_1(5)\) and $X_{1}(7)$ have genus zero, so the corresponding
\(\mathbb Q\)-rational cuspidal divisor class groups are trivial. More
generally, concise formulas for \(|L_N(k)|\) are obtained for $5 \le N \le 10$ and $N=12$ when
\(X_1(N)\) has genus zero; see Theorem~\ref{genuszero}. As an application of the general polytope characterization for exponent vectors $\mathbf a$, we derive counting formulas for larger \(N\) where \(X_1(N)\) has nontrivial genus. For $N \ge 5$, the counts decompose into
$$
|L_{N}(k)|=T_{N}(k)+\mathcal{E}_{N}(k),
$$
where $T_{N}(k) \in \Bbb Q[k]$, called the \emph{polynomial part} of $|L_{N}(k)|$, has degree~$\lfloor N/2\rfloor$ in $k$, and $\mathcal{E}_{N}(k)$, called the \emph{quasipolynomial part} of $|L_{N}(k)|$, is a piecewise polynomial function of degree smaller than $\lfloor N/2\rfloor$ in~$k$. Both components depend on the structure and order of the \(\mathbb Q\)-rational cuspidal divisor class group of
\(X_1(N)\). For example, in level \(N=11\), whose corresponding $\mathbb{Q}$-rational cuspidal divisor class group is $\mathbb{Z}/5\mathbb{Z}$, Theorem~\ref{level11} gives
\begin{align}
\label{L11}
|L_{11}(k)| = \frac{25 k^5}{24} + \frac{125 k^4}{24} + \frac{75 k^3}{8} + \frac{175 k^2}{24} + \frac{137 k}{60} + \frac{1}{5} + \mathcal{E}_{11}(k),
\end{align}
where 
\[
\mathcal{E}_{11}(k)
=\begin{cases}
4/5, & k\equiv0\pmod 5,\\
-2/5, & k\equiv1\pmod 5,\\
2/5, & k\equiv2\pmod 5,\\
-4/5, & k\equiv3\pmod 5,\\
0, & k\equiv4\pmod 5.
\end{cases}
\]

The remainder of the present work is organized as follows. Section \ref{def} introduces basic notions, including a formula for the presentation matrix of the rational cuspidal divisor class group of \(X_1(N)\). In Section~3, we use modularity criteria for Klein forms to determine conditions on the exponents appearing in \eqref{eq:fa-definition1} such that $f^{(N)}_{\mathbf a}(\tau)$ is a modular form of level $N$. For each fixed $k\geq0$, we show that exponents $\mathbf a\in L_{N}(k)$ are contained within an $\lfloor N/2 \rfloor$-dimensional polytope and satisfy congruences determined by invariant factors of the rational cuspidal divisor class group. In Section~$4$, this construction is used to develop effective strategies for deducing formulas for $|L_{N}(k)|$. In Section~5, these strategies are illustrated by a number of explicit examples, including~\eqref{L11}.

\section{Foundations} \label{def}
Define
\[
\Gamma_1(N)
=
\left\{
\begin{pmatrix}
a & b\\
c & d
\end{pmatrix}
\in{\rm SL}_{2}(\mathbb{Z})
:
a\equiv d\equiv 1 \pmod N,\;
c\equiv 0 \pmod N
\right\}.
\]
A holomorphic (meromorphic) modular form of weight $k$ on $\Gamma_1(N)$ is a holomorphic (meromorphic) function
$f$ on the upper half plane $\mathbb{H}$ and satisfying the transformation
\[
f\!\left(\frac{a\tau+b}{c\tau+d}\right)=(c\tau+d)^k f(\tau)
\qquad
\text{for all }
\begin{pmatrix}
a & b\\
c & d
\end{pmatrix}\in \Gamma_1(N),
\]
and which is holomorphic (meromorphic) at all cusps. A modular function for $\Gamma_1(N)$ is a meromorphic function on $\mathbb{H}$ and at the cusps that is invariant under action by $\Gamma_1(N)$. To define the behavior of $f$ at cusps, let \(\mathfrak c\) be a cusp of \(\Gamma_1(N)\), and choose \(\gamma\in{\rm SL}_{2}(\Bbb Z)\) such that
\(\gamma(\infty)=\mathfrak c\). Let \(h_{\mathfrak c}\) denote the width of \(\mathfrak c\), i.e., the smallest positive integer such that $\gamma \left ( \begin{smallmatrix} 1 & h_{\mathfrak c} \\ 0 & 1 \end{smallmatrix} \right ) \gamma^{-1} \in \Gamma_1(N)$. In particular, for cusps $\mathfrak c = [a/c]$ of $\Gamma_1(N)$, the cusp width is $h_{\mathfrak c} = N/\gcd(c,N)$ \cite[Proposition 6.3.20]{cohen}. A local
parameter at the cusp \(\mathfrak c\) is $
q_{\mathfrak c}=\exp\!\left(2\pi i \tau/h_{\mathfrak c}\right)$, and a nonzero meromorphic modular form \(f\) of weight \(k\) on \(\Gamma_1(N)\) has a Laurent expansion in the local parameter
\[
(f|_k\gamma)(\tau)=\sum_{n=n_0}^\infty a_n q_{\mathfrak c}^{\,n},
\qquad a_{n_0}\neq 0.
\]
In this case, the order of \(f\) in terms of the local parameter at \(\mathfrak c\) is defined by $\Ord_{\mathfrak c}(f)=n_0$. 
As before, let $M_k(\Gamma_1(N))$ denote the space of holomorphic modular forms of weight $k$ on $\Gamma_1(N)$. Let $Y_1(N) =\Gamma_1(N)\backslash \mathbb{H}$, and denote the corresponding compact modular curve $
X_1(N)=\Gamma_1(N)\backslash \bigl(\mathbb{H}\cup \mathbb{P}^1(\mathbb{Q})\bigr).$ The cusps are $\Gamma_1(N)$-equivalence classes of points of $\mathbb{P}^1(\mathbb{Q})$, written $[a/c]$, for $\gcd(a,c)=1$. A complete set of cusp representatives of $X_{1}(N)$ is given by 
\begin{align} \label{cusp}
\left\{\left[a/(gu)\right]:\,
g\mid N,\,
u\in (\mathbb{Z}/(N/g)\mathbb{Z})^\times,\,
a\in (\mathbb{Z}/g\mathbb{Z})^\times/\{\pm1\}, \, \gcd(a,gu)=1\right\}.
\end{align}
If $\mathfrak c=[a/c]$ is such a cusp, define
\begin{align}
\label{gcusp}
    g = g(\mathfrak c)=\gcd(c,N).
\end{align}
Let \(X=X_1(N)\) and write $\mathfrak C_{N}$ for the set of cusps of $X_{1}(N)$ that are  defined \cite{stevens} over $\mathbb{Q}(\omega_{N})$, where $\omega_{N}=e^{2\pi i/N}$. 
The Galois group
$
{\rm Gal}(\mathbb{Q}(\omega_N)/\mathbb{Q})\cong (\mathbb{Z}/N\mathbb{Z})^\times
$ acts on the rational cusps via 
\begin{align}
\sigma_x[a/(gu)]=[a/x(gu)], \qquad x\in (\mathbb Z/N\mathbb Z)^\times.  \label{galois}
\end{align}
See, e.g., \cite[Section 1.3]{stevens}.
Thus the Galois orbit of a cusp \(\mathfrak c = [a/(gu)]\) is obtained by allowing \(u\) to run
through all units modulo \(N/g\), while \(g\) and the class of \(a\) modulo \(g\)
remain fixed. Hence the Galois orbit of a cusp from \eqref{cusp} is determined by the ordered pair \((g,a)\), for $g \mid N$ and $a\in (\mathbb Z/g\mathbb Z)^\times/\{\pm1\}$.

\begin{lemma}\label{lem:types}
The number of Galois orbits of cusps for $X_1(N)$ via \eqref{galois} is $\lfloor N/2 \rfloor+1$.
\end{lemma}

\begin{proof}
For each divisor $d\mid N$ the number of $a\in(\mathbb Z/d\mathbb Z)^\times/\{\pm 1\}$
is $1$ if $d=1$ or $d=2$, and equals $\varphi(d)/2$ if $d>2$.
Thus the number of pairs $(d,a)$ is
\begin{align}
    \sum_{d \mid N} |(\Bbb Z / d \Bbb Z)^{\times}/ \{\pm 1\}|=1+\mathbf{1}_{2 \mid N}+ \sum_{d\mid N \atop d > 2} \frac{\varphi(d)}{2} = \lfloor N/2 \rfloor +1,
\end{align}
where the penultimate equality follows from $\sum_{d \mid N} \varphi(d) = N$. 
\end{proof}

Let $m= \lfloor N/2 \rfloor$. Define the group of cuspidal divisors
\[
\operatorname{Div}^c(X)=\bigoplus_{\mathfrak c\in  \mathfrak C_{N}}\mathbb Z[\mathfrak c],
\]
and let
\[
\operatorname{Div}^c_{\mathbb Q}(X)=\left\{D\in\operatorname{Div}^c(X):\,\sigma(D)=D\,\,\mbox{for any $\sigma\in{\rm Gal}(\mathbb{Q}(\omega_{N})/\mathbb{Q})$}\right\}
\]
be its $\mathbb Q$-rational part, i.e., the subgroup generated by the sums of cusps that are invariant under the Galois action \eqref{galois}. For $f$ a nonzero meromorphic function on  \(X\), define its divisor
\[
\operatorname{div}(f)=\sum_{P\in X} \Ord_P(f)\,[P],
\]
and the degree $\deg:\operatorname{div}(X)\to \mathbb Z$ of a divisor as the sum of the orders. By Lemma~\ref{lem:types}, there are $m+1$ Galois orbit classes. Choose representatives from each Galois orbit class
$\
\mathfrak c_0,\mathfrak c_1,\dots,\mathfrak c_m
$
with $\mathfrak c_{m} = [1/N]= [i\infty]$. 
For $0 \le r \le m$, let \(\Orb( \mathfrak c_r)\) denote the corresponding Galois orbit, and note that $|\Orb(\mathfrak c_{m})| =1$. Denote the orbit sum
\[
C_r=\sum_{\mathfrak c\in \Orb(\mathfrak c_r)} [\mathfrak  c].
\]
Then each \(C_r\) is a \(\mathbb{Q}\)-rational cuspidal divisor on \(X\), and \(C_0,\dots,C_m\) form a \(\mathbb{Z}\)-basis of \(\operatorname{Div}^c_{\mathbb{Q}}(X)\). Write
$
\mu_r=\deg(C_r)=|\Orb(\mathfrak{c}_{r})|.
$
Then the degree-zero subgroup \(\operatorname{Div}_{\mathbb{Q}}^{0,c}(X)\) of the rational cuspidal divisors is free of rank \(m\), and a basis is
\begin{equation}
\label{FR}
\mathcal F_r=C_r-\mu_r C_m,
\qquad
0\le r\le m-1.
\end{equation}
A modular unit on $X$ is a modular function with divisor supported only at the cusps. Let
\[
\mathcal{U}_{\mathbb{Q}}
=
\{u\in \mathbb{Q}(X)^\times : \operatorname{div}(u)\text{ is supported on cusps}\}/\mathbb{Q}^\times.
\]
A principal divisor is a divisor that is the image of a nonzero meromorphic function. Since principal divisors on $X$ have degree zero, the restriction of the divisor map to $\mathcal{U}_{\mathbb{Q}}$ is an injective homomorphism. Its image is the group of principal divisors denoted by
\[
\operatorname{Prin}^c_{\mathbb Q}(X)
=
\operatorname{im}\bigl(\operatorname{div}:\mathcal U_{\mathbb Q}\to \operatorname{Div}^{0,c}_{\mathbb Q}(X)\bigr).
\]
The rational cuspidal divisor class group is defined as
\[
\Cl^c_{\mathbb{Q}}(X)
=
\operatorname{Div}_{\mathbb{Q}}^{0,c}(X)\big/\Prin^c_{\mathbb{Q}}(X).
\]
Let $u_1,\dots,u_m$ be a $\Bbb Z$-basis of $\mathcal U_{\mathbb Q}$. With respect to the basis $\mathcal F_0,\dots, \mathcal F_{m-1}$ of $\operatorname{Div}^{0,c}_{\mathbb Q}(X)$, each divisor $\operatorname{div}(u_i)$ can be written uniquely in terms of the class matrix $V_{N} = (V_{r,i})_{1 \le r, i \le m}$ via
\begin{align} \label{clV}
\operatorname{div}(u_i)=\sum_{r=0}^{m-1} V_{r+1,i}\mathcal F_r,
\qquad V_{r,i}\in \mathbb Z.
\end{align}
The principal divisors and the class group may be represented in terms of the class matrix
\[
\operatorname{Prin}^c_{\mathbb Q}(X)=V_N\mathbb Z^m
\subset \mathbb Z^m\cong \operatorname{Div}^{0,c}_{\mathbb Q}(X), \quad \text{and} \quad \mathrm{Cl}^c_{\mathbb Q}(X)\cong \mathbb Z^m/V_N\mathbb Z^m.
\]

To describe the class matrix explicitly, we first construct a basis for the lattice determining the product of powers of generalized eta functions, 
\begin{align} \label{siegel}
q^{\frac{N}{2}B_2(i/N)}(q^i,q^{N-i};q^N)_\infty, \qquad B_{2}(x) = x^{2} -x+ \frac{1}{6},
\end{align}
that generate the group of modular units.

\begin{lemma}[\cite{Str}] \label{ylem}
Let \(N\ge 5\) and
\begin{equation}\label{PIN}
    \Pi_{N}=
\left\{
(e_1,\dots,e_m)\in \mathbb Z^m:
\sum_{j=1}^m e_j\equiv 0 \pmod{12},
\ \sum_{j=1}^m j^2 e_j\equiv 0 \pmod{\gcd(2,N)N}
\right\},
\end{equation}
Then the set of the products
\[
\prod_{j=1}^m q^{\frac{Ne_{j}}{2}B_2(j/N)}(q^j,q^{N-j};q^N)_\infty^{e_{j}},
\qquad (e_{1},\ldots,e_{m})\in \Pi_{N},
\]
taken modulo multiplication by elements of \(\mathbb Q^\times\), is exactly the set of modular units on \(X_1(N)\) with \(\mathbb Q\)-rational cuspidal divisors and in particular, is a free abelian group of rank~$m$. 
\end{lemma}

We apply Lemma~\ref{ylem} to find a \(\mathbb Z\)-basis for the lattice \(\Pi_{N}\). This will be used to construct the matrix generating the divisor class group over $\Bbb Z$.
\begin{lemma}\label{prop:explicit-BN}
Let \(m=\lfloor N/2\rfloor\) and \(M=N\gcd(2,N)\), and denote
\begin{align} \label{ab}
\alpha(\mathbf a)=\sum_{j=1}^m a_j,
\qquad
\beta(\mathbf a)=\sum_{j=1}^m j^2a_j.
\end{align}

For \(N=5\), define
\[
B_5=
\begin{pmatrix}
16 & -5\\
-4 & 5
\end{pmatrix}.
\]

Now assume \(N\ge 6\). Let \(r_N\in\{0,1,2\}\) be the unique integer such that
\[
r_N\equiv 2M \pmod 3,
\]
and define
\[
a_N=\frac{-M+5r_N}{3},
\qquad
b_N=\frac{M-8r_N}{3}.
\]
Set
\[
\mathbf z_1=(16,-4,0,\dots,0)^T,
\qquad
\mathbf z_2=(a_N,b_N,r_N,0,\dots,0)^T.
\]

For each \(j\in\{3,\dots,m\}\) with \(3\nmid j\), define
\[
\mathbf w_j=
\left(\frac{j^2-4}{3},\ -\frac{j^2-1}{3},\ 0,\dots,0,1,0,\dots,0\right)^T,
\]
where the \(1\) occurs in the \(j\)-th coordinate.

Let
\[
J_0=\{j\in\{3,\dots,m\}:3\mid j\}=\{j_1<\dots<j_r\}.
\]
If there is at least one multiple of $3$ in $J_0$, i.e., \(r\ge 1\), define
\[
\mathbf v_0=\left(j_1^2-4,\ -(j_1^2-1),\ 0,\dots,0,3,0,\dots,0\right)^T,
\]
where the entry \(3\) is in the \(j_1\)-th coordinate, and for \(2\le t\le r\),
\[
\mathbf v_t=
\left(\frac{j_t^2-j_1^2}{3},\ -\frac{j_t^2-j_1^2}{3},\ 0,\dots,0,-1,0,\dots,0,1,0,\dots,0\right)^T,
\]
where the entries \(-1\) and \(1\) occur in the \(j_1\)-th and \(j_t\)-th coordinates, respectively.
Denote
\[
\mathcal B_N=
\{\mathbf z_1,\mathbf z_2\}
\cup
\{\mathbf w_j:3\le j\le m,\ 3\nmid j\}
\cup
\bigl(\{\mathbf v_0\}\cup\{\mathbf v_t:2\le t\le r\}\bigr).
\]
Then $\mathcal{B}_{N}$ is a $\mathbb{Z}$-basis for $\Pi_{N}$ defined as in~\eqref{PIN}.
\end{lemma}

\begin{proof}
We first treat the case \(N\ge 6\), so \(m\ge 3\).

Set
\[
K=\{\mathbf u\in\mathbb Z^m:\alpha(\mathbf u)=0,\ \beta(\mathbf u)=0\}.
\]
We will prove that the vectors
\[
\{\mathbf w_j:3\nmid j,\ 3\le j\le m\}
\cup
\bigl(\{\mathbf v_0\}\cup\{\mathbf v_t:2\le t\le r\}\bigr)
\]
form a \(\mathbb Z\)-basis of \(K\). Let \(\mathbf u=(u_1,\dots,u_m)^T\in K\). Since \(\alpha(\mathbf u)=\beta(\mathbf u)=0\),
\[
u_1+u_2+\sum_{j=3}^m u_j=0,
\qquad
u_1+4u_2+\sum_{j=3}^m j^2u_j=0.
\]
Subtracting gives
\[
3u_2+\sum_{j=3}^m (j^2-1)u_j=0,
\]
hence
\[
u_2=-\frac13\sum_{j=3}^m (j^2-1)u_j,
\qquad
u_1=\frac13\sum_{j=3}^m (j^2-4)u_j.
\]
Therefore \(\mathbf u\in K\) if and only if the coordinates \(u_3,\dots,u_m\in\mathbb Z\) satisfy the condition
\[
\sum_{j=3}^m (j^2-1)u_j\equiv 0 \pmod 3.
\]
Now \(j^2\equiv 1\pmod 3\) if \(3\nmid j\), and \(j^2\equiv 0\pmod 3\) if \(3\mid j\), so this condition becomes
\[
\sum_{3\mid j,\ 3\le j\le m} u_j\equiv 0 \pmod 3.
\]

Thus \(K\) is the subgroup of \(\mathbb Z^{m-2}\) consisting of tuples \((u_3,\dots,u_m)\) such that the sum of the coordinates indexed by multiples of \(3\) is divisible by \(3\). A \(\mathbb Z\)-basis for this subgroup can be formulated in terms of elementary vectors:
\[
\mathbf e_j \quad (3\nmid j),
\qquad
3 \mathbf e_{j_1},
\qquad
\mathbf e_{j_t}-\mathbf e_{j_1}\quad (2\le t\le r).
\]

Substituting these basis vectors in the formulas for \(u_1\) and \(u_2\) yields exactly the vectors \(\mathbf w_j\), \(\mathbf v_0\), and \(\mathbf v_t\). Indeed, for \(3\mathbf e_{j_1}\) one gets
\[
u_1=\frac13(j_1^2-4)\cdot 3=j_1^2-4,
\qquad
u_2=-\frac13(j_1^2-1)\cdot 3=-(j_1^2-1),
\]
which gives the stated formula for \(\mathbf v_0\). Hence these vectors form a \(\mathbb Z\)-basis of \(K\).

Next, we show that \(\mathbf z_1,\mathbf z_2\in \Pi_{N}\). A direct calculation gives
\[
\alpha(\mathbf z_1)=12,\qquad \beta(\mathbf z_1)=0, \qquad \alpha(\mathbf z_2)=a_N+b_N+r_N=0,
\]
while
\[
\beta(\mathbf z_2)=a_N+4b_N+9r_N
=\frac{-M+5r_N}{3}+4\frac{M-8r_N}{3}+9r_N=M.
\]
So \(\mathbf z_1,\mathbf z_2\in\Pi_{N}\). Also every basis vector of \(K\) lies in \(\Pi_{N}\), since \(K\subseteq\Pi_{N}\).

Now let \(\mathbf a\in\Pi_{N}\). Write
\(
\alpha(\mathbf a)=12s,\) \(\beta(\mathbf a)=Mt
\)
for some \(s,t\in\mathbb Z\), and set
\(
\mathbf a'=\mathbf a-s\mathbf z_1-t\mathbf z_2.
\)
Then
\[
\alpha(\mathbf a')=\alpha(\mathbf a)-12s=0,
\qquad
\beta(\mathbf a')=\beta(\mathbf a)-Mt=0,
\]
so \(\mathbf a'\in K\). Since our chosen vectors form a \(\mathbb Z\)-basis of \(K\), it follows that \(\mathbf a\) is a \(\mathbb Z\)-linear combination of the vectors in \(\mathcal B_N\). Hence \(\mathcal B_N\) spans \(\Pi_{N}\).

It remains to prove linear independence. Suppose
\[
s\mathbf z_1+t\mathbf z_2+\mathbf k=0,
\]
where \(\mathbf k\in K\) is a \(\mathbb Z\)-linear combination of the chosen basis of \(K\). Applying \(\alpha\) gives \(12s=0\), hence \(s=0\). Applying \(\beta\) then gives \(Mt=0\), hence \(t=0\). Therefore \(\mathbf k=0\), and since our chosen vectors form a basis of \(K\), all coefficients are zero. Thus the vectors in \(\mathcal B_N\) are \(\mathbb Z\)-linearly independent. Therefore \(\mathcal B_N\) is a \(\mathbb Z\)-basis of \(\Pi_{N}\) for \(N\ge 6\).

For \(N=5\), let \(\mathbf x=(16,-4)^T\) and \(\mathbf y=(-5,5)^T\). Then
\[
\alpha(\mathbf x)=12,\qquad \beta(\mathbf x)=0,
\qquad
\alpha(\mathbf y)=0,\qquad \beta(\mathbf y)=15.
\]
So \(\mathbf x,\mathbf y\in\Pi_5\). If \(\mathbf a=(a_1,a_2)\in\Pi_5\), write
\(\alpha(\mathbf a)=12s,\) \(\beta(\mathbf a)=5t
\)
with \(s,t\in\mathbb Z\). Then \(\mathbf a-s\mathbf x=(\delta,-\delta)\) for some \(\delta\in\mathbb Z\), and
\(
\beta(\mathbf a-s\mathbf x)=3\delta
\)
must be divisible by \(5\), so \(\delta\in 5\mathbb Z\). Hence \(\mathbf a-s\mathbf x\in\mathbb Z\mathbf y\). Thus \(\mathbf x,\mathbf y\) span \(\Pi_5\). Since
\(
\det(B_5)=60\ne 0,
\)
they are linearly independent. Therefore the columns of \(B_5\) form a \(\mathbb Z\)-basis of \(\Pi_5\).
\end{proof}

For $z=(Q_{1},Q_{2})\in\mathbb{Q}^{2}-\mathbb{Z}^{2}$ and $q_{z}=e^{2\pi i(Q_{1}\tau+Q_{2})}$, a Klein form $K_{(Q_{1},Q_{2})}(\tau)$ is  defined by
$$
K_{(Q_{1},Q_{2})}(\tau)=e^{\pi iQ_{2}(Q_{1}-1)}q^{\frac{1}{2}Q_{1}(Q_{1}-1)}(1-q_{z})\prod_{n=1}^{\infty}(1-q_{z}q^{n})(1-q_{z}^{-1}q^{n})(1-q^{n})^{-2}.
$$
These, along with the Dedekind eta function $\eta(\tau) = q^{1/24}\prod_{n=1}^{\infty}(1-q^{n})$, are useful in determining modular properties of the products $f_{\mathbf a}^{(N)}$.

\begin{lemma}[\cite{KL}]\label{klf}
Let $K_{(Q_{1},Q_{2})}(\tau)$ be defined as above. Then the following assertions hold.
\begin{enumerate}

\item For $(Q_{1},Q_{2})\in\mathbb{Q}^{2}-\mathbb{Z}^{2}$ and $(s_{1},s_{2})\in\mathbb{Z}^{2}$, one has
  \begin{align*}
    K_{(-Q_{1},-Q_{2})}(\tau)&=-K_{(Q_{1},Q_{2})}(\tau) \\
    K_{(Q_{1},Q_{2})+(s_{1},s_{2})}(\tau) &=(-1)^{s_{1}s_{2}+s_{1}+s_{2}}e^{-\pi i(s_{1}Q_{2}-s_{2}Q_{1})}K_{(Q_{1},Q_{2})}(\tau).
  \end{align*}
\item For any $\begin{pmatrix}a&b\\ c&d\end{pmatrix}\in {\rm SL}_{2}(\mathbb{Z})$, 
  \begin{align*}
K_{(Q_{1},Q_{2})}\left(\frac{a\tau+b}{c\tau+d}\right)=(c\tau+d)^{-1}K_{(Q_{1}a+Q_{2}c,Q_{1}b+Q_{2}d)}(\tau).     
  \end{align*}
\item 
The order of vanishing of $K_{(Q_{1},Q_{2})}(\tau)$ at the cusp $i\infty$ is given by
\begin{align*}
\Ord_{\infty}(K_{(Q_{1},Q_{2})})=\frac{1}{2}\langle Q_{1}\rangle\left(\langle Q_{1}\rangle -1\right),   
\end{align*}
where $\langle r\rangle=r-\lfloor r\rfloor$ denotes the fractional part of~$r$.
\end{enumerate}
\end{lemma} 

The order formula in the next lemma is an immediate consequence of the classical transformation formula for $\eta(\tau)$ and Parts (2) and (3) of Lemma \ref{klf}.
\begin{lemma} \label{ord}
Suppose $a_{1},\ldots,a_{m}\in\mathbb{Z}$ are such that  \[\displaystyle f = \prod_{i=1}^{m}\eta(N\tau)^{2a_i}K_{(i/N,0)}(N\tau)^{a_i} \] is a meromorphic modular function for $\Gamma_1(N)$. Then, the order at the cusp $\mathfrak{c}=[a/c]$ in terms of its local parameter is  $$\Ord_{[a/c]} \left( f \right)  = \frac{g(\mathfrak c)}{2}\ \sum_{i=1}^{m} a_{i}B_2\!\Bigl(\Big\langle \frac{ia}{g(\mathfrak c)}\Big\rangle\Bigr),$$
where $g(\mathfrak{c})$ is defined by~\eqref{gcusp}, and $\langle r\rangle=r-\lfloor r\rfloor$ denotes the fractional part of~$r\in\mathbb{R}$.
\end{lemma}

\begin{lemma}\label{lem:ANBN-equals-VN}
Let $\mathfrak c_0, \mathfrak c_1, \ldots, \mathfrak c_{m}$ be representatives of the Galois orbits determined by the action \eqref{galois} on the cusps of $X$. For each cusp $\mathfrak c=[a/c]$, with $g(\mathfrak c)$ defined by~\eqref{gcusp}, define the vector
\begin{align} \label{beta}
\beta_{\mathfrak c}
=
\Bigl(\beta_{\mathfrak c,1},\dots,\beta_{\mathfrak c,m}\Bigr)\in \mathbb Q^m,
\qquad
\beta_{\mathfrak c,i}=\frac{g(\mathfrak c)}{2}\,B_2\!\Bigl(\Big\langle \frac{ia}{g(\mathfrak c)}\Big\rangle\Bigr),
\end{align} and construct a matrix from the row vectors corresponding to the cusps $\mathfrak c_0, \ldots \mathfrak c_{m-1}$
\begin{align} \label{val}
A_N=\bigl(\beta_{\mathfrak{c}_{r},j}\bigr)_{0\le r\le m-1,\ 1\le j\le m}.
\end{align}
Let $B_{N} = (b_{i,j})$ be the matrix formed by the column vectors of $\mathcal{B}_{N}$ in Lemma \ref{prop:explicit-BN}.
If $V_{N}$ is the class matrix defined by \eqref{clV}, then
\[
V_N=A_NB_N.
\]
In particular, $A_N$ has rank $m$.
\end{lemma}

\begin{proof}
Fix \(i\in\{1,\dots,m\}\).  
For each \(0\le r\le m-1\), by Lemma \ref{ord} and the definition of $B_{N}$, 
\[
\Ord_{\mathfrak c_r}(u_i)=\sum_{j=1}^m b_{j,i}\,\Ord_{\mathfrak c_r}(\eta(N\tau)^{2}K_{(j/N,0)}(N\tau)).
\]
By the definition of \(\beta_{\mathfrak c_r,j}\), the order of \(\eta(N\tau)^{2}K_{(j/N,0)}(N\tau)\) at \(\mathfrak c_r\) is
\(
\beta_{\mathfrak c_r,j}.
\)
Therefore
\[
\Ord_{\mathfrak c_r}(u_i)=\sum_{j=1}^{m} \beta_{\mathfrak c_r,j} b_{j,i} = (A_NB_N)_{r+1,i}. 
\]
By definition of \(V_N\),
\[
\operatorname{div}(u_i)=\sum_{r=0}^{m-1}(V_N)_{r+1,i}\mathcal F_r.
\]
Using \(\mathcal F_r=C_r-\mu_r C_m\), this becomes
\[
\operatorname{div}(u_i)
=
\sum_{r=0}^{m-1}(V_N)_{r+1,i}C_r
-
\left(\sum_{r=0}^{m-1}\mu_r (V_N)_{r+1,i}\right) C_m.
\]
Hence, for each \(0\le r\le m-1\), the coefficient of \(C_r\) in \(\operatorname{div}(u_i)\) is
\(
(V_N)_{r+1,i},
\)
and this agrees with the order of \(u_i\) at the cusp orbit represented by \(\mathfrak{c}_r\), namely 
\[
(V_N)_{r+1,i}=\Ord_{\mathfrak c_r}(u_i) = (A_NB_N)_{r+1,i}. 
\]

Finally, the last assertion on $A_{N}$ in the lemma follows from combining Lemmas~\ref{ylem} and~\ref{ord}.
\end{proof}

\section{The products $f^{(N)}_{\mathbf a}(\tau)$ and their polytopes}

 In this section, we  derive a characterization of $L_N(k)$ defined as in~\eqref{LNA} in terms of the $\mathbb{Q}$-rational cuspidal divisor class group. This description of the lattice will be useful for counting the polytopes corresponding to $L_N(k)$. 

We include the following elementary congruence to aid the reader in translating the level $N$ modularity conditions from \cite{KL} to those in this work.
\begin{lemma} \label{eqv}
Let $S_{2}(\mathbf{a})$ be defined as in~\eqref{S1S2}. For all \(N\ge5\),
\[
S_2(\mathbf a)\equiv0\pmod{N\gcd(2,N)} \quad \Longleftrightarrow \quad
\sum_{j=1}^{m}j^2a_j\equiv0\pmod{N\gcd(2,N)}.
\]
\end{lemma}

\begin{proof}
Since
\[
S_2(\mathbf a)
=
\sum_{j=1}^{m}j(N-j)a_j
=
N\sum_{j=1}^{m}ja_j
-
\sum_{j=1}^{m}j^2a_j,
\]
the claim of the lemma is clearly true when \(N\) is odd. When \(N\) is even, we work modulo \(2N\). The difference between
\(S_2(\mathbf a)\) and \(-\sum j^2a_j\) is
\(
N\sum_{j=1}^{m}ja_j.
\)
This term is divisible by \(2N\) exactly when
\(
\sum_{j=1}^{m}ja_j\equiv0\pmod2.
\)
The required claim follows from $j \equiv j^{2} \pmod{2}.$
\end{proof}

\begin{lemma}\label{thef}
For a cusp representative \(\mathfrak c=[a/c]\), let
\(
g=g(\mathfrak c)
\)
be defined by~\eqref{gcusp}.
Define
\begin{align}\label{ty}
v_{[a/c]}\!\left(f^{(N)}_{\mathbf a}\right)
=
\frac{g}{24}a_0
+
\frac{g}{2}\sum_{j=1}^{m}
a_j
B_2\left(\left\langle \frac{ja}{g}\right\rangle\right),
\qquad
B_2(x)=x^2-x+\frac16.
\end{align}
Then
\(
f^{(N)}_{\mathbf a}\in M_{a_0/2}(\Gamma_1(N))
\)
if and only if
\(
v_{[a/c]}\!\left(f^{(N)}_{\mathbf a}\right)\ge0
\)
for every cusp \([a/c]\in\mathfrak C_N\), and
\begin{equation}\label{wh}
S_1(a_0,\mathbf a)\equiv 0\pmod{24},
\qquad
S_{2}(\mathbf a)\equiv0\pmod{N\gcd(2,N)}.
\end{equation}
\end{lemma}

\begin{proof}
   From the definition of Klein forms,  one can verify that
\begin{align}
\label{fak}
f_{{\bf a}}^{(N)}(\tau)=\eta(N\tau)^{S_{1}(a_0,{\bf a})}\prod_{i=1}^{m}K_{(i/N,0)}(N\tau)^{a_{i}},
\end{align}
where in particular,
\begin{align} \label{kfp}
\prod_{i=1}^{m}K_{(i/N,0)}(N\tau)^{a_{i}}\in q^{-\frac{1}{2N}S_{2}({\bf a})}(1+O(q)).
\end{align}

When $f_{{\bf a}}^{(N)}(\tau)\in M_{a_{0}/2}(\Gamma_{1}(N))$, then $\Ord_{[a/c]} f^{(N)}_{\mathbf a}(\tau)\geq0$ for any cusp $[a/c]$.  In particular, the
order at the cusp \([0/1]\) is a nonnegative integer, so that
\begin{align} \label{alr}
\Ord_{[0/1]}\!\left(f_{\mathbf a}^{(N)}\right)
=
\frac{S_1(a_0,\mathbf a)}{24} \in \Bbb Z.
\end{align}
Since the order of \(f_{\mathbf a}^{(N)}\) at infinity must also be a nonnegative integer, 
\[
\frac{N}{24}S_1(a_0,\mathbf a)-\frac{1}{2N}S_2(\mathbf a)\in\mathbb Z.
\]
By \eqref{alr}, it follows that
\(
\frac{1}{2N}S_2(\mathbf a)\in\mathbb Z.
\)
Note that, regardless of the parity of $N$, 
\(
N\gcd(2,N)\mid 2N.
\)
Therefore \(
S_2(\mathbf a)\equiv0\pmod{2N}
\) implies
\(
S_2(\mathbf a)\equiv0\pmod{N\gcd(2,N)}.
\)

For the converse, note that $j(N-j)$ is even when $N$ is odd. Thus, it suffices to assume $\frac{S_{1}(a_0,{\bf a})}{24}$ and $\frac{1}{2N}S_{2}({\bf a})$ are integral. The factor
\(
\eta(N\tau)^{S_1(a_0,\mathbf a)}
\)
has trivial multiplier when
\(
S_1(a_0,\mathbf a)\equiv 0\pmod{24}.
\) Therefore, the expressions \eqref{fak}, \eqref{kfp} may be used with Lemma~\ref{klf}(1) and (2), and Lemma \ref{eqv} (c.f., \cite[p.~68]{KL}) to show that $f_{{\bf a}}^{(N)}(\tau)$ is a modular form of weight~$a_{0}/2$ and level $\Gamma(N)$. The nonnegativity of the expression $v_{[a/c]}\!\left(f^{(N)}_{\mathbf a}\right)$ in~\eqref{ty} makes $f_{{\bf a}}^{(N)}(\tau)$ holomorphic. Since $\Gamma_{1}(N)$ is generated by $\Gamma(N)$ and $\left ( \begin{smallmatrix}
    1&1\\0&1
\end{smallmatrix}\right )$, condition~\eqref{wh} guarantees $f_{{\bf a}}^{(N)}(\tau)$ to be invariant under the action of $\left ( \begin{smallmatrix}
    1&1\\0&1
\end{smallmatrix} \right )$. Therefore, $f_{{\bf a}}^{(N)}(\tau)$ must be a holomorphic modular form of weight~$a_{0}/2$ and level $\Gamma_{1}(N)$.
\end{proof}

The following auxiliary lemma is straightforward to verify.

\begin{lemma}\label{lem:psolution}
Let $\Pi_{N}$ be defined by \eqref{PIN}. For each \(s\in\mathbb Z\), define
\(\mathbf a^{(s)}\in\mathbb Z^m\) with $m=\lfloor N/2\rfloor$ by
$$
\mathbf{a}^{(s)}=\begin{cases}
-3s\,\mathbf e_1
+
2s\,\mathbf e_2 = (-3s,2s)&\mbox{for $N=5$,}\\
    -3s\,\mathbf e_1
+
3s\,\mathbf e_2
-
s\,\mathbf e_3 = (-3s, 3s, -s, 0 \cdots, 0)&\mbox{for $N\geq6$.}
\end{cases}
$$
Then if \(s\equiv k\pmod{12}\), the vector \(\mathbf a^{(s)}=(a_{1}^{(s)},\ldots, a_{m}^{(s)})\) satisfies
\[
\sum_{j=1}^{m}a_j^{(s)}\equiv -k\pmod{12}, \qquad
\sum_{j=1}^{m}j^2a_j^{(s)}
\equiv0\pmod{N\gcd(2,N)}.
\]
Moreover, the coset \(\mathbf a^{(s)}+\Pi_N\) is independent of the choice of
representative \(s\) modulo \(12\).
\end{lemma}

What follows is one of the main results of the present work characterizing the central object $L_{N}(k)$, the set of exponents $\mathbf{a}$ determining a holomorphic modular form $f_{\mathbf a}^{(N)}$.

\begin{theorem} \label{mt} Let $A_N$ and $B_N$ be defined as in Lemma \ref{lem:ANBN-equals-VN}, and let  $g(\mathfrak{c})$ be defined by~\eqref{gcusp} for a cusp $\mathfrak{c}$.
Suppose \(
P_{N}V_{N}Q_{N}=\operatorname{diag}(d_1,\dots,d_m)
\)
is the Smith normal form of the class matrix $V_N = A_N B_N$. Define $\mathbf a^{(k)}$ as in Lemma \ref{lem:psolution}. Then $|L_{N}(k)|$ is finite, with
\[
L_N(k)
=
\left\{
A_N^{-1}\mathbf w(\mathbf v)\ \middle|\
\begin{array}{l}
v_{\mathfrak c_r} \in \Bbb Z^{+}\cup \{0\},\quad 0\le r\le m\\[1mm] 
\sum_{r=0}^{m} |\Orb(\mathfrak c_{r})| v_{\mathfrak c_{r}} =
 \frac{k}{24}  \sum_{d\mid N} d\,\varphi(d)\,\varphi(N/d)\\[1mm]
\bigl(P_N(\mathbf w(\mathbf v)-A_N\mathbf a^{(k)})\bigr)_r \equiv 0 \pmod{d_r},
\ 1\le r\le m
\end{array}
\right\},   
\]
where 
\[
\mathbf v=(v_{\mathfrak c_0},\dots,v_{\mathfrak c_{m-1}})^T, \quad \quad \mathbf g=(g(\mathfrak c_0),\dots,g(\mathfrak c_{m-1}))^T, \quad \mathbf w(\mathbf v)=\mathbf v-\frac{k}{12}\mathbf g.
\]
\end{theorem}
\begin{proof}
Define
\[
\mathcal C_N(k)
=
\left\{
\mathbf a\in\mathbb Z^m:
\sum_{j=1}^{m}a_j\equiv -k\pmod{12},
\quad
\sum_{j=1}^{m}j^2a_j\equiv0\pmod{M}
\right\},
\]
where $M=N\gcd(2,N)$.
By Lemma~\ref{lem:psolution}, the vector \(\mathbf a^{(k)}\) lies in
\(\mathcal C_N(k)\). We claim that
\[
\mathcal C_N(k)=\mathbf a^{(k)}+\Pi_{N},
\]
where $\Pi_{N}$ is defined as in~\eqref{PIN}.
If \(\mathbf a\in\mathcal C_N(k)\), then
\[
\mathbf a-\mathbf a^{(k)}\in\Pi_{N}.
\]
Conversely, if \(\mathbf b\in\Pi_{N}\), then
\(
\mathbf a^{(k)}+\mathbf b
\)
satisfies the two affine congruences defining \(\mathcal C_N(k)\). Hence
\[
\mathcal C_N(k)=\mathbf a^{(k)}+\Pi_{N}.
\]

By Lemma~\ref{prop:explicit-BN}, the columns of \(B_N\) form a
\(\mathbb Z\)-basis for \(\Pi_{N}\). Hence
\[
\Pi_{N}=B_N\mathbb Z^m.
\]
Therefore
\[
A_N\Pi_{N}=A_NB_N\mathbb Z^m=V_N\mathbb Z^m.
\]

We now prove the stated characterization of $L_{N}(k)$. First, suppose
\(
\mathbf a\in L_N(k).
\)
 Since \(f_{\mathbf a}^{(N)}\) is
holomorphic, we have
\[
v_{\mathfrak c_r}\ge0
\quad\mbox{for}\quad 
0\le r\le m.
\]
The valence formula implies
\begin{align} \label{valence}
\sum_{r=0}^{m}
|\operatorname{Orb}(\mathfrak c_r)|v_{\mathfrak c_r}
 =
 \frac{k}{12}[{\rm SL}_2(\Bbb Z):\pm{\Gamma_1(N)}]= 
\frac{k}{24}
\sum_{d\mid N}d\varphi(d)\varphi(N/d).
\end{align}
Also,
\[
\mathbf w(\mathbf v)=A_N\mathbf a.
\]
By Lemmas \ref{eqv} and \ref{thef},  \(\mathbf a\in\mathcal C_N(k)=\mathbf a^{(k)}+\Pi_{N}\), so there exists
\(\mathbf b\in\Pi_{N}\) such that
\(
\mathbf a=\mathbf a^{(k)}+\mathbf b.
\)
Thus
\[
\mathbf w(\mathbf v)-A_N\mathbf a^{(k)}
=
A_N\mathbf b.
\]
Since \(A_N\mathbf b\in A_N\Pi_{N}=V_N\mathbb Z^m\), we have
\[
\mathbf w(\mathbf v)-A_N\mathbf a^{(k)}
\in V_N\mathbb Z^m.
\]
Applying \(P\) and using the Smith normal form
\[
P_{N} V_N Q_{N}=\operatorname{diag}(d_1,\dots,d_m),
\]
we get
\[
P_{N}\left(\mathbf w(\mathbf v)-A_N\mathbf a^{(k)}\right)
\in
\operatorname{diag}(d_1,\dots,d_m)\mathbb Z^m.
\]
Hence every element of \(L_N(k)\) is a vector satisfying the stated
conditions.

Conversely, suppose that \(\mathbf v\) is a nonnegative integer vector whose component values, together with the nonnegative integer \(v_{\mathfrak c_m}\), satisfy the valence equation \eqref{valence} and
\[
P_{N}\left(\mathbf w(\mathbf v)-A_N\mathbf a^{(k)}\right)
\in
\operatorname{diag}(d_1,\dots,d_m)\mathbb Z^m.
\]
This is equivalent to
\[
\mathbf w(\mathbf v)-A_N\mathbf a^{(k)}
\in V_N\mathbb Z^m.
\]
Therefore there exists \(\mathbf n\in\mathbb Z^m\) such that
\[
\mathbf w(\mathbf v)-A_N\mathbf a^{(k)}
=
V_N\mathbf n.
\]
Since \(V_N=A_NB_N\), we get
\[
\mathbf w(\mathbf v)
=
A_N\mathbf a^{(k)}+A_NB_N\mathbf n
=
A_N(\mathbf a^{(k)}+B_N\mathbf n).
\]
Hence
\[
A_N^{-1}\mathbf w(\mathbf v)
=
\mathbf a^{(k)}+B_N\mathbf n.
\]
Because \(B_N\mathbf n\in\Pi_{N}\), we have
\[
A_N^{-1}\mathbf w(\mathbf v)\in
\mathbf a^{(k)}+\Pi_{N}
=
\mathcal C_N(k).
\]
Thus the exponent vector
\(
\mathbf a=A_N^{-1}\mathbf w(\mathbf v)
\)
satisfies the required congruences for the fixed value \(a_0=2k\). Form the product \(f_{\mathbf a}^{(N)}\) using the exponent vector $\mathbf a$, so, by Lemmas \ref{eqv} and \ref{thef}, $f_{\mathbf a}^{(N)}$ is weakly modular for $\Gamma_1(N)$, and hence has integral orders at cusps and satisfies the valence equation. Note that the first \(m\) cusp orders of \(f_{\mathbf a}^{(N)}\) agree with the integers
\(v_{\mathfrak c_0},\dots,v_{\mathfrak c_{m-1}}\), because
\[
(\Ord_{\mathfrak c_{0}} (f_{\mathbf a}^{(N)}), \ldots,\Ord_{\mathfrak c_{m-1}} (f_{\mathbf a}^{(N)}))  =\frac{k}{12}\mathbf g+A_N\mathbf a
=
\frac{k}{12}\mathbf g+\mathbf w(\mathbf v)
=
\mathbf v.
\]
The final cusp order at \(\mathfrak c_m=[i\infty]\) is determined by the
valence formula and equals \(v_{\mathfrak c_m}\). Let
\(
x=\Ord_{\mathfrak c_m}(f_{\mathbf a}^{(N)})
\)
be the cusp order at $\mathfrak c_{m}$. Since \(f_{\mathbf a}^{(N)}\) is weakly modular,
the valence formula gives
\begin{align} \label{vala}
\sum_{r=0}^{m-1}|\Orb(\mathfrak c_r)|v_{\mathfrak c_r}
+
|\Orb(\mathfrak c_m)|x
=
\frac{k}{24}\sum_{d\mid N}d\varphi(d)\varphi(N/d)
.\end{align}
On the other hand, by assumption, if we replace $x$ with $v_{\mathfrak c_{m}}$ in \eqref{vala}, the statement remains true. Subtracting the two identities yields $|\Orb(\mathfrak c_m)|(x-v_{\mathfrak c_m})=0.$ Since \(|\Orb(\mathfrak c_m)|=1\), it follows that $x=v_{\mathfrak c_m}$.  Therefore all cusp orders of
\(f_{\mathbf a}^{(N)}\) are nonnegative. Hence
\[
\mathbf a=A_N^{-1}\mathbf w(\mathbf v)\in L_N(k).
\]
This proves the equality of the two sets. The final conclusion in the theorem follows from the valence formula together with the nonnegativity of $v_{\mathfrak{c}}$.
\end{proof}

\section{Computing $|L_{N}(k)|$}

In this section, using the characterization of $L_{N}(k)$ given in Theorem~\ref{mt}, we develop effective strategies for deducing general formulas for the cardinality $|L_{N}(k)|$.

\subsection{Construction of polytope count}

Based on the characterization of $L_{N}(k)$ given by Theorem~\ref{mt}, one can immediately formulate $|L_{N}(k)|$ in terms of polytope counting functions.
\begin{theorem}
For $N\geq5$ and $k\geq0$, let $L_{N}(k)$ be defined as in~\eqref{LNA}. Let $P_{N}V_{N}Q_{N} = D_{N}:={\rm diag}(d_{1},\ldots,d_{m})$ be a Smith normal form of the divisor class matrix $V_{N}$, with congruences written \begin{align} \label{smith3}
P_N\mathbf v+\boldsymbol\alpha_N k
\equiv \mathbf 0 \pmod{D_N}, \qquad \boldsymbol\alpha_N
=
-P_N\left(\frac{1}{12}\mathbf g+A_N\mathbf a^{(1)}\right)
=
(\alpha_1,\dots,\alpha_m)^T . 
\end{align} Then one has
\begin{align}\label{thel}
|L_{N}(k)|
=
\frac{1}{d_1\cdots d_m}
\sum_{j_1=0}^{d_1-1}\cdots\sum_{j_m=0}^{d_m-1}
\xi(\mathbf j)^k\,W_{\mathbf j}(kR_{N}),
\end{align}
where
\begin{align} \label{xi}
\xi(\mathbf j)
=
\prod_{i=1}^{m}\omega_i^{j_i\alpha_i}, \qquad 
\mathbf j=(j_1,\dots,j_m),\ 
0\le j_i\le d_i-1, \qquad \omega_i=e^{2\pi i/d_i},
\end{align}
\begin{align} \label{chi}
\chi_r(\mathbf j)
=
\prod_{i=1}^{m}
\omega_i^{j_i(P_N)_{i,r+1}}, \qquad 0\le r\le m-1, \qquad \chi_m(\mathbf j)=1,
\end{align}
and 
\begin{align}\label{thew-class}
W_{\mathbf j}(n)
=
\sum_{\substack{
(v_{\mathfrak c_0},\dots,v_{\mathfrak c_m})\in\mathbb Z_{\ge0}^{m+1}\\
\sum_{r=0}^{m}\mu_r v_{\mathfrak c_r}=n
}}
\prod_{r=0}^{m}
\chi_r(\mathbf j)^{v_{\mathfrak c_r}}, \quad \mu_i=
\begin{cases}
\dfrac{\varphi(N)}{2}, & (g_i,a_i)=(1,1),\\[2mm]
\dfrac{\varphi(N/2)}{2}, & (g_i,a_i)=(2,1)\ \text{if }2\mid N,\\[2mm]
\varphi(N/g_i), & g_i>2.
\end{cases}
\end{align}
\end{theorem}
\begin{proof}
For the list of \((g,a)\) with $g\mid N,\ a\in (\mathbb Z/g\mathbb Z)^\times/\{\pm1\}$, 
fix an ordering
\(
\{(g_i,a_i):0\le i\le m\},
\) with $g_{m} = N$, and let $\mathfrak c_{i}$ be the corresponding cusp class.
Let \(\mu_i= |\Orb (\mathfrak c_{i})|.
\)
Then, based on the discussion in the proof of Lemma \ref{lem:types}, the orbit sizes are given by $\mu_i$ as defined in \eqref{thew-class}.
With this convention, the valence formula becomes
\begin{align} \label{val1}
\sum_{r=0}^{m}\mu_r\,v_{\mathfrak c_r}=k\,R_N, \qquad  R_N=\frac{1}{24}\sum_{d\mid N} d\,\varphi(d)\,\varphi(N/d).
\end{align}

Let $\mathbf v = (v_{\mathfrak c_0},\dots,v_{\mathfrak c_{m-1}})$, with the cusp orders ordered so that
\(
\mathfrak c_m=[i\infty]
\).
Let
\begin{align} \label{smith1}
P_NV_NQ_N=D_N=\operatorname{diag}(d_1,\dots,d_m),
\qquad
P_N,Q_N\in \operatorname{GL}_m(\mathbb Z),
\end{align}
be a Smith normal form of the class matrix \(V_N\). The class-matrix congruences may be written
\begin{align} \label{smith2}
P_N\mathbf v+\boldsymbol\alpha_N k
\equiv \mathbf 0 \pmod{D_N}, 
\end{align}
where
\(
\boldsymbol\alpha_N=(\alpha_1,\dots,\alpha_m)^T\in\mathbb Z^m
\) 
defined by \eqref{smith3} records the part of the congruence dependent on $k$ and 
is integral since  $P_{N}$ is unimodular and \(f_{\mathbf a^{(1)}}^{(N)}\) with \(a_0=2\) is weakly holomorphic. 
For
\(
\mathbf j=(j_1,\dots,j_m),
0\le j_i\le d_i-1,
\)
define $\xi(\mathbf j)$ and $\chi_r(\mathbf j)$ as in \eqref{xi}, \eqref{chi}. By applying the indicator function for the $i$th congruence,
\[
\frac1{d_i}\sum_{j_i=0}^{d_i-1}
\omega_i^{j_i n}
=
\begin{cases}
1, & n\equiv0\pmod{d_i},\\
0, & n\not\equiv0\pmod{d_i}
\end{cases}, \qquad \omega_i=e^{2\pi i/d_i},
\]
we obtain, for a fixed vector
\(
(v_{\mathfrak c_0},\dots,v_{\mathfrak c_m})\in\mathbb Z_{\ge0}^{m+1},
\)
the indicator that all congruences hold is
\begin{align} \label{allc}
\frac{1}{d_1\cdots d_m}
\sum_{j_1=0}^{d_1-1}\cdots\sum_{j_m=0}^{d_m-1}
\prod_{i=1}^{m}
\omega_i^{j_i\bigl((P_N\mathbf v)_i+\alpha_i k\bigr)}.
\end{align}
Factors in the innermost product of \eqref{allc} may be collected to get
\begin{align*}
\prod_{i=1}^{m}
\omega_i^{j_i\bigl((P_N\mathbf v)_i+\alpha_i k\bigr)}
=
\left(
\prod_{i=1}^{m}\omega_i^{j_i\alpha_i}
\right)^k
\prod_{r=0}^{m-1}
\left(
\prod_{i=1}^{m}\omega_i^{j_i(P_N)_{i,r+1}}
\right)^{v_{\mathfrak c_r}} =\xi(\mathbf j)^k
\prod_{r=0}^{m}
\chi_r(\mathbf j)^{v_{\mathfrak c_r}}.
\end{align*}
Applying the valence formula in the form \eqref{val1} to \eqref{allc} yields
\begin{align} \label{lo}
|L_N(k)|
&=
\sum_{\substack{
(v_{\mathfrak c_0},\dots,v_{\mathfrak c_m})\in\mathbb Z_{\ge0}^{m+1}\\
\sum_{r=0}^{m}\mu_r v_{\mathfrak c_r}=kR_{N}
}}
\frac{1}{d_1\cdots d_m}
\sum_{j_1=0}^{d_1-1}\cdots\sum_{j_m=0}^{d_m-1}
\xi(\mathbf j)^k
\prod_{r=0}^{m}
\chi_r(\mathbf j)^{v_{\mathfrak c_r}}.
\end{align}
Interchanging the finite sums in \eqref{lo} gives \eqref{thel}.
\end{proof}

\subsection{Decomposition}
It is well known that polytopes may be enumerated with quasipolynomials. A quasipolynomial of degree $d$ and period $E$ is a function of the form $$Q(k) = c_{d}(k)k^{d} + \cdots + c_{0}(k),$$ where the coefficients $c_{i}(k)$ are periodic functions of period $E$. Formula \eqref{thel} provides an expansion for the quasipolynomial for the count $|L_N(k)|$ but requires substantial algebra to simplify in practice for general levels $N$ as the number of cusps and the size of the class group increase. Write
\[
|L_N(k)|
=
c_m(k)k^m+c_{m-1}(k)k^{m-1}+\cdots+c_0(k),
\]
where each \(c_i(k)\) is periodic in $k$ with period \(E\). We will decompose the quasipolynomial into a main polynomial term and a lower-order quasipolynomial term whose coefficient function depends on the residue class of $k\pmod E$. 
Define
\begin{align} \label{deft}
T_N(k)
=
\sum_{i=0}^{m}\overline c_i k^i, \qquad \overline c_i
=
\frac1E\sum_{r=0}^{E-1}c_i(r),
\end{align}
and 
\begin{align} \label{defe}
\mathcal E_{N}(k)
=\begin{cases}
\sum_{i=0}^{m}
\bigl(c_i(0)-\overline c_i\bigr)k^i&\mbox{if $k\equiv 0\pmod{E}$,}\\
\qquad\qquad\vdots&\vdots\\
\sum_{i=0}^{m}\bigl(c_i(E-1)-\overline c_i\bigr)k^i&\mbox{if $k\equiv E-1\pmod{E}$.}
\end{cases}
\end{align}
Clearly, one has for any integer~$k$,
$$
|L_{N}(k)|=T_{N}(k)+\mathcal{E}_{N}(k).
$$
Under this convention, $T_N(k) \in \Bbb Q[k]$ is independent of the residue class of $k$ modulo $E$, and $\mathcal E_{N}(k)$ is a piecewise polynomial function.

Theorem \ref{lt} subsequently shows that $c_{m}(0)=\cdots=c_{m}(E-1)$, and so the degree of $T_{N}(k)$ is exactly $m=\lfloor N/2\rfloor$. The leading coefficient is the leading term of the quasipolynomial for $|L_{N}(k)|$ and is expressed in terms of the index of $\pm\Gamma_1(N)$ in ${\rm SL}_2(\Bbb Z)$, the sizes of the Galois orbits, and the order of the class group. Since the leading coefficient in the quasipolynomial for $|L_N(k)|$ is constant,  $$\deg \mathcal E_{N} < \deg T_N =m.$$ Here $\deg \mathcal{E}_{N}$ is the highest exponent among all the monomial terms in the constituent polynomials of the quasipolynomial $\mathcal{E}_{N}$. Our proof of Theorem \ref{lt} requires a fundamental property of rational polytopes due to Ehrhart \cite{ehr} that characterizes the counting function for a polytope as a quasipolynomial and determines its leading coefficient (c.f., \cite[\S 5.4]{beckrobins}, \cite[\S 4]{stanleyEC1}). 

\begin{lemma}
\label{lem:ehr}
Let \(V\) be a real vector space of dimension \(d\), let
\(\Lambda\subset V\) be a rank-\(d\) lattice, and let \(\mathcal{P}\subset V\) be a
\(d\)-dimensional rational polytope with respect to \(\Lambda\). Then
\(
|k\mathcal{P}\cap \Lambda|
\) is a quasipolynomial in \(k\) of degree
\(d\) with constant leading coefficient
\(
\operatorname{vol}_{\Lambda}(\mathcal{P})
\).
\end{lemma}
We denote by \(\operatorname{vol}_{\Lambda}\) the volume of a polytope measured relative to the
lattice \(\Lambda\), so that
 \(\operatorname{vol}_{\Lambda}(P) \)
 is the volume of $P$ divided by the volume of
\(\Pi_{\Lambda}\), where $\Pi_{\Lambda}$ is a fundamental parallelepiped for $\Lambda$.

\begin{theorem} \label{lt}
Let \(N\ge5\) and $k\geq0$, and \(m=\lfloor N/2\rfloor\). Then
\[
|L_N(k)| = \frac{R_N^m}
{m!\,\mu_0\mu_1\cdots\mu_m\,|\det V_N|} k^{m}+O(k^{m-1}),
\]
where as before, $V_{N}$ is the class matrix for the $\mathbb{Q}$-rational cuspidal divisor class group of $\Gamma_{1}(N)$ defined by~\eqref{clV}, $\mu_{i}=|{\rm Orb}(\mathfrak{c}_{i})|$, and 
$
R_N= 
\frac{1}{24}\sum_{d\mid N} d\,\varphi(d)\,\varphi(N/d).
$
\end{theorem}

\begin{proof}
Theorem~\ref{mt} demonstrates the bijection
\begin{align*}
L_{N}(k)&\cong\left\{
\mathbf v=(v_{0},\ldots,v_{m-1})\in \mathbb{Z}^{m} \middle|\
\begin{array}{l}
v_{r}\geq0,\quad 0\le r\le m-1\\[1mm] 
\sum_{i=0}^{m-1} \mu_{i} v_{i} \leq
 k R_{N},\\[1mm]
\mathbf{v}-\frac{k}{12}\mathbf{g}-A_N\mathbf a^{(k)}\in V_{N}\mathbb{Z}^{m}
\end{array}
\right\}.
\end{align*}
Let $\Lambda_{N}=V_{N}\mathbb{Z}^{m}$ and define
$$
\mathcal{P}_{N}=\left\{
(y_0,\dots,y_{m-1})\in\mathbb R^{m}:
y_{0}\geq0,\,\ldots,y_{m-1}\geq0,\,\sum_{i=0}^{m-1}\mu_i y_i\le R_N
\right\}.
$$
Since translation preserves cardinality of the corresponding cosets,  the following set bijections hold:
$$
L_{N}(k)\cong k\mathcal{P}_{N}\cap \left(\frac{k}{12}\mathbf{g}+A_N\mathbf a^{(k)}+\Lambda_{N}\right)\cong k\left( \mathcal{P}_{N}-\frac{1}{12}\mathbf{g}-A_N\mathbf a^{(1)}\right)\cap\Lambda_{N}.
$$
Since the polytopes
\(\mathcal P_N-\frac1{12}\mathbf g-A_N\mathbf a^{(1)}\) and \(\mathcal P_N-\frac1{12}\mathbf g\) differ by a fixed translation, they have the same relative volume. Hence the quasipolynomials counting the number of lattice points in their dilations by $k$ have the same leading coefficient. 
Therefore, the preceding bijections imply that as quasipolynomials, $|L_{N}(k)|$ and $\left|k\left(\mathcal{P}_{N}-\frac{1}{12}\mathbf{g}\right)\cap\Lambda_{N}\right|$ have the same leading term. Invoking Lemma~\ref{lem:ehr} for $V=\mathbb{R}^{m}$, $\Lambda=\Lambda_{N}$ and $\mathcal{P}=\mathcal{P}_{N}-\frac{1}{12}\mathbf{g}$, one can tell that the leading term is ${\rm vol}_{\Lambda_{N}}(\mathcal{P}_{N}-\frac{1}{12}\mathbf{g})k^{m}$. The relative volume may be computed from elementary calculus, and one finds that
\begin{equation*}
{\rm vol}_{\Lambda_{N}}\left(\mathcal{P}_{N}-\frac{1}{12}\mathbf{g}\right)={\rm vol}_{\Lambda_{N}}(\mathcal{P}_{N})=\frac{R_N^m}
{m!\,\mu_0\mu_1\cdots\mu_m\,|\det V_N|}.  \qedhere
\end{equation*}
\end{proof}

\subsection{Generating functions}

It is well known that lattice-point generating functions for rational polyhedra are rational functions; see \cite[Chapter~13]{barvinok} and \cite[Chapter 4]{stanleyEC1}. The next theorem obtains such expansions for the generating function of $|L_{N}(k)|$ by using the fact that the generating function for $W_{\mathbf j}(n)$ from \eqref{thew-class} may be expressed as the product
\begin{align}\label{gen}
G_{\mathbf j}(x)
=
\sum_{n\ge0}W_{\mathbf j}(n)x^n
=
\prod_{r=0}^{m}
\frac{1}{1-\chi_r(\mathbf j)x^{\mu_r}}.
\end{align}
In particular, since the poles of the generating function are roots of unity (c.f., \cite[Lemma 3.24]{beckrobins}), we may derive an upper bound on the period of the quasipolynomial for $|L_{N}(k)|$  in terms of the \emph{exponent} of the cuspidal divisor class group, defined as the least common multiple of the invariant factors of $V_{N}$, $$E_{N} = \min \{ E \ge 1\ :\ E\cdot x =0\text{ for every } x\in \Bbb Z^{m} / V_N \Bbb Z^{m} \} = \operatorname{lcm}(d_1,\dots,d_m).$$
\begin{theorem} \label{genform}
Denote the generating function for $|L_N(k)|$ by\begin{equation}
\label{FN}
F_N(q)=\sum_{k\ge 0}|L_N(k)|\,q^k.
\end{equation}
\begin{enumerate}
\item
Then $F_{N}(q)$ is a rational function over $\Bbb Q$. 

\item
If $F_N(q) = P(q)/Q(q)$ with $\gcd (P(q), Q(q)) = 1$ and the denominator factors $Q(q)=\prod_{t=1}^{T}(1-\rho_t q)^{e_t},$ where $\rho_{t}$ is a root of unity, then \(|L_N(k)|\) is a quasipolynomial in $k$ of period dividing $\displaystyle \lcm_{1 \le t \le T} \Ord (\rho_{t})$, where $\Ord(\rho)$ is the multiplicative order of $\rho$. 

\item
In particular, if $\mu_{i} = |\Orb(\mathfrak c_{i})|$, the order of the rational Galois orbit of the cusp class $\mathfrak{c}_{i}$, and $R_N = \frac{1}{24}\sum_{d\mid N} d\,\varphi(d)\,\varphi(N/d)$, then the quasipolynomial period divides
\[
\operatorname{lcm}
\left(
E_N,\,
\frac{\mu_0E_N}{\gcd(\mu_0E_N,R_N)},\,
\dots,\,
\frac{\mu_mE_N}{\gcd(\mu_mE_N,R_N)}
\right).
\]
When $N= p \ge5$ is prime, the period of the quasipolynomial divides $$\frac{12E_p}{\gcd(p+1,12)}.$$

\end{enumerate}
\end{theorem}

\begin{proof}
Part (1) is a property of quasipolynomials proven in \cite[Proposition 4.4.1]{stanleyEC1}.

To prove part (2), recall from \eqref{gen} that the generating function for the coefficients $W_{\mathbf j}(n)$ from \eqref{thew-class} is given by the product $G_{\mathbf j}(x)$. By \eqref{thel}, the generating function for $F_{N}(q)$ is given by \[F_N(q)
=
\frac{1}{d_1\cdots d_m}
\sum_{\mathbf j} H_{\mathbf j}(q), \qquad H_{\mathbf j}(q)
:=
\sum_{k\ge0}\xi(\mathbf j)^k W_{\mathbf j}(kR_N)q^k.\] Since each \(G_{\mathbf j}(x)\) has denominator
\(
\prod_{r=0}^{m}(1-\chi_r(\mathbf j)x^{\mu_r}),
\)
every pole of \(F_N(q)\) occurs at a root of unity. Therefore \(F_N(q)\) can be written in the form
\(
F_N(q)=\frac{P(q)}{Q(q)},
\)
where all roots of \(Q(q)\) are roots of unity. Factor \(Q(q)\) over \(\mathbb C\) as
\[
Q(q)=\prod_{t=1}^{T}(1-\rho_t q)^{e_t},
\qquad \rho_t \text{ a root of unity}.
\]
Then the partial fraction decomposition of \(F_N(q)\) has the form
\[
F_N(q)=\sum_{t=1}^{T}\sum_{m=1}^{e_t}\frac{A_{t,m}}{(1-\rho_t q)^m}.
\]
Now
\[
\frac{1}{(1-\rho_{t} q)^m}
=
\sum_{k\ge 0}\binom{k+m-1}{m-1}\rho_{t}^k q^k,
\]
so, if $[q^{k}]$ denotes the coefficients of $q^{k}$,
\begin{align} \label{pt}
[q^k]F_N(q)
=
\sum_{t=1}^{T}\sum_{m=1}^{e_t}
A_{t,m}\binom{k+m-1}{k}\rho_t^k.
\end{align}
Since \(\rho_t\) is a root of unity, each summand is a polynomial in \(k\) times a periodic function of \(k\). Hence \([q^k]F_N(q) = |L_N(k)|\) is a quasipolynomial in \(k\). 

Finally, the period divides the least common multiple of the respective orders of the roots of unity $\rho_{t}$ that occur. 
By the construction of \eqref{thel}, \(\xi(\mathbf j)\) and
\(\chi_r(\mathbf j)\) are roots of unity involving products of powers of \(\omega_i=e^{2\pi i/d_i}\). Therefore, every root of unity that appears in
\(\xi(\mathbf j)\) or in some \(\chi_r(\mathbf j)\) has order dividing
\(
E_N \).
Thus
\[
\xi(\mathbf j)^{E_N}=1,
\qquad
\chi_r(\mathbf j)^{E_N}=1
\]
for every \(\mathbf j\) and every \(r\). Now fix \(\mathbf j\) and consider the rational function $G_{\mathbf j}(x)$ from \eqref{gen}.
Each denominator factor of $G_{\mathbf j}(x)$ has the form
\(
1-\chi_r(\mathbf j)x^{\mu_r},
\) where $\mu_r$ are the sizes of the Galois orbits.
The roots \(\rho\) satisfy
\(
\rho^{\mu_r}=\chi_r(\mathbf j)^{-1}.
\)
Since \(\chi_r(\mathbf j)^{E_N}=1\), we get
\(
\rho^{\mu_rE_N}=1.
\)
Therefore every root \(\rho\) of this factor is a root of unity whose order
divides
\(
\mu_rE_N.
\) 
If \(\rho\) has order dividing \(\mu_rE_N\), then the period of
$\rho^{kR_{N}}$
divides
\[
\frac{\mu_rE_N}{\gcd(\mu_rE_N,R_N)}.
\]
Note that the factors
\(
\xi(\mathbf j)^k
\)
have orders dividing \(E_N\). Therefore every periodic factor appearing in the
coefficient formula has period dividing
\begin{align} \label{di}
\operatorname{lcm}
\left(
E_N,\,
\frac{\mu_0E_N}{\gcd(\mu_0E_N,R_N)},\,
\dots,\,
\frac{\mu_mE_N}{\gcd(\mu_mE_N,R_N)}
\right).
\end{align}

Now suppose $N=p \ge 5$ is prime. The Galois orbit sizes for prime level are
\[
\mu_0=\frac{p-1}{2},
\qquad
\mu_1=\cdots=\mu_m=1, \quad \text{and} \quad R_{p}=\frac{p^2-1}{24} = \mu_0\cdot\frac{p+1}{12}.
\]
By \eqref{di}, the quasipolynomial period divides
\[
\operatorname{lcm}
\left(
E_p,\,
\frac{\mu_0 E_p}{\gcd(\mu_0 E_p,R_p)},\,
\frac{E_p}{\gcd(E_p,R_p)} 
\right) = \operatorname{lcm}
\left(
E_p,\,
\frac{\mu_0 E_p}{\gcd(\mu_0 E_p,R_p)}
\right) \mid \frac{\mu_0 E_{p}}{\gcd(\mu_0,R_p)}.
\]
Since
\(
R_p=\mu_0\frac{p+1}{12},
\)
the number
\[
\frac{\mu_0}{\gcd(\mu_0,R_p)}
\]
is the denominator of
\[
\frac{R_p}{\mu_0}=\frac{p+1}{12}
\]
in lowest terms. Therefore
\[
\frac{\mu_0}{\gcd(\mu_0,R_p)}
=
\frac{12}{\gcd(p+1,12)}.
\]
This proves the final claim of the theorem.
\end{proof}

\subsection{Computing quasipolynomial constituents by interpolation and recursion}

Beck and Robins \cite[Lemma 3.14]{beckrobins} expand lattice counting functions in terms of binomial coefficients. Theorem \ref{int} similarly interpolates the polynomial constituents of $|L_N(k)|$ in terms of binomial coefficients and, for Galois orbit orders \(\boldsymbol\mu=(\mu_0,\dots,\mu_m)\), the weighted complete
homogeneous polynomial
\[
h_n^{\boldsymbol\mu}(x_0,\dots,x_m)
= 
\sum_{\substack{\mu_0e_0+\cdots+\mu_m e_m=n\\e_{0},\,\ldots,\,e_{m}\geq0}}
x_0^{e_0}\cdots x_m^{e_m}.
\]
\begin{theorem} \label{int} Suppose $E$ is a period of the quasipolynomial for $|L_N(k)|$. For $0\le r<E$,
\begin{equation}\label{eq:general-binomial-piece}
|L_N(k)|
=
\sum_{\ell=0}^{m}
p_{r,\ell}
\binom{t+m-\ell}{m}, \qquad k=Et+r,
\end{equation}
where 
\begin{align*}
p_{r,s} &= \sum_{a=0}^{s}
(-1)^a\binom{m+1}{a}|L_N(E(s-a)+r)| \\ 
&=
\frac1{|\det V_N|}
\sum_{j_1=0}^{d_1-1}\cdots\sum_{j_m=0}^{d_m-1}
\sum_{a=0}^{s}
(-1)^a\binom{m+1}{a}
\xi(\mathbf j)^{E(s-a)+r}
h_{R_N(E(s-a)+r)}^{\boldsymbol\mu}
\bigl(
\chi_0(\mathbf j),\dots,\chi_m(\mathbf j)
\bigr).
\end{align*}
\end{theorem}
\begin{proof}
Fix a residue class \(r\pmod E\), and write \(k=Et+r\). Since \(E\) is a
period of the quasipolynomial for \(|L_N(k)|\), Lemma \ref{lt} implies the sequence
\(
b_{r,t}=|L_N(Et+r)|
\)
is a polynomial in \(t\) of degree \(m\). The polynomials
\(
\binom{t+m}{m},\binom{t+m-1}{m},\dots,\binom{t}{m}
\)
form a basis for the vector space of polynomials in \(t\) of degree at most
\(m\). Hence $b_{r,t}$ has a unique
expansion 
\[
b_{r,t}
=
\sum_{\ell=0}^{m}
p_{r,\ell}\binom{t+m-\ell}{m}.
\]
Let $B_r(z)=\sum_{t\ge0}b_{r,t}z^t$. Then 
\[
B_r(z)=\frac{P_r(z)}{(1-z)^{m+1}},
\qquad 
P_r(z)=\sum_{\ell=0}^{m}p_{r,\ell}z^\ell.
\]
Taking the coefficient of \(z^s\) in $P_r(z)=(1-z)^{m+1}B_r(z)$ gives
\begin{align} \label{tr}
p_{r,s}
=
\sum_{a=0}^{s}
(-1)^a\binom{m+1}{a}b_{r,s-a}
=
\sum_{a=0}^{s}
(-1)^a\binom{m+1}{a}|L_N(E(s-a)+r)|.
\end{align}
From \eqref{thel},
\begin{align} \label{pi}
|L_N(u)|
=
\frac{1}{d_1\cdots d_m}
\sum_{j_1=0}^{d_1-1}\cdots\sum_{j_m=0}^{d_m-1}
\xi(\mathbf j)^u
h_{R_N u}^{\boldsymbol\mu}
\bigl(\chi_0(\mathbf j),\dots,\chi_m(\mathbf j)\bigr).
\end{align}
Taking \(u=E(s-a)+r\) in \eqref{pi} and applying \eqref{tr} gives the stated formula.
\end{proof}
Theorem \ref{int} says that if we can compute either of the two expressions for the coefficients $p_{r,s}$ in Theorem \ref{int}, then the quasipolynomial constituents for each congruence class $k = Et +r$ may be written down explicitly. In practice, the class group size and structure make the final expression of Theorem \ref{int} difficult to compute. It is more efficient in most cases to derive the first expression for $p_{r,s}$ in terms of a finite number of direct counts for $|L_N(k)|$. For this purpose, we introduce a recursive technique for computing the polytope counts. The recursion may be broken into pieces and computed in parallel, making the computation of the polynomial constituents for larger $N$ feasible.  
\begin{theorem} \label{recur}
Let the cusps $\mathfrak{c}_{i}$ be ordered so that
\(
\mathfrak c_m=[i\infty],
\)
and denote 
\(
\mathbf v=(v_{\mathfrak c_0},\dots,v_{\mathfrak c_{m-1}})^T
\). Let
\[
P_NV_NQ_N=D_N=\operatorname{diag}(d_1,\dots,d_m),
\qquad
P_N,Q_N\in\operatorname{GL}_m(\mathbb Z),
\]
be a Smith normal form of the class matrix \(V_N\). Write the congruences in \eqref{smith1}--\eqref{smith2} as
\[
P_N\mathbf v+\boldsymbol\alpha_N k
\equiv \mathbf 0 \pmod{D_N}.
\]
Let
\(
\mathcal R_N
=
\mathbb Z/d_1\mathbb Z
\times\cdots\times
\mathbb Z/d_m\mathbb Z.
\) For \(0\le i\le m-1\), let
\(
\boldsymbol\gamma_i
\)
be the \((i+1)\)-th column of \(P_N\), viewed as an element of \(\mathcal R_N\), and set
\(
\boldsymbol\gamma_m=\mathbf 0\in\mathcal R_N.
\)
Also set
\(
\boldsymbol\beta=\boldsymbol\alpha_N\in\mathcal R_N.
\)
For \(0\le j\le m+1\), \(S\in\mathbb Z\), and
\(\boldsymbol\rho\in\mathcal R_N\), define \(D_j(S,\boldsymbol\rho)\) by
\[
D_0(S,\boldsymbol\rho)
=
\begin{cases}
1, & S=0 \text{ and } \boldsymbol\rho=\mathbf 0,\\
0, & \text{otherwise,}
\end{cases}
\]
and, for \(1\le j\le m+1\), by the recursion
\[
D_j(S,\boldsymbol\rho)
=
D_{j-1}(S,\boldsymbol\rho)
+
D_j(S-\mu_{j-1},\boldsymbol\rho-\boldsymbol\gamma_{j-1}),
\]
with
\(
D_j(S,\boldsymbol\rho)=0
\) whenever $S<0$.
Then, for every \(k\ge0\),
\[
|L_N(k)|=
D_{m+1}(kR_{N},-k\boldsymbol\beta).
\]
\end{theorem}

\begin{proof}
We first prove that the recursion counts the number of lattice points satisfying the congruence conditions. For $0\le j\le m+1$, let $C_{j}(S,\boldsymbol{\rho})$ be the number of tuples $(v_{\mathfrak{c}_{0}},\ldots,v_{\mathfrak{c}_{j-1}})\in\mathbb{Z}_{\ge0}^{j}$ satisfying $$\sum_{i=0}^{j-1}\mu_{i}v_{\mathfrak{c}_{i}}=S,\qquad\sum_{i=0}^{j-1}\boldsymbol{\gamma}_{i}v_{\mathfrak{c}_{i}}=\boldsymbol{\rho}\quad\text{in }\mathcal{R}_{N}.$$ For $j=0$, there is one empty tuple of total $0$ and residue $\mathbf{0}$, and no other empty tuples. Hence $C_{0}=D_{0}$. 

Assume the interpretation is known for smaller totals at the same stage and for the previous stage. Fix $j\ge1$. Partition the tuples counted by $C_{j}(S,\boldsymbol{\rho})$ according to the last coordinate $v_{\mathfrak{c}_{j-1}}$. If $v_{\mathfrak{c}_{j-1}}=0$, the first $j-1$ coordinates have the same total and the same residue, so these tuples are counted by $D_{j-1}(S,\boldsymbol{\rho})$. If $v_{\mathfrak{c}_{j-1}}>0$, subtracting $1$ from $v_{\mathfrak{c}_{j-1}}$ gives a tuple using the same j variables, but with total $S-\mu_{j-1}$ and residue $\boldsymbol{\rho}-\boldsymbol{\gamma}_{j-1}$. This operation is reversible by adding $1$ back to the last coordinate. Thus the tuples with positive last coordinate are counted by $$D_{j}(S-\mu_{j-1},\boldsymbol{\rho}-\boldsymbol{\gamma}_{j-1}).$$ 
The two cases are disjoint and exhaustive, so
\[
C_j(S,\boldsymbol\rho)
=
D_{j-1}(S,\boldsymbol\rho)
+
D_j(S-\mu_{j-1},\boldsymbol\rho-\boldsymbol\gamma_{j-1}).
\]
Thus, by induction on \(j\) and \(S\),
\(
D_j(S,\boldsymbol\rho)=C_j(S,\boldsymbol\rho)
\) for all \(j,S,\boldsymbol\rho\). Taking \(j=m+1\), the value
\(
D_{m+1}(kR_{N},\boldsymbol\rho)
\)
counts all nonnegative cusp-order vectors satisfying the valence equation
\(
\sum_{i=0}^{m}\mu_i v_{\mathfrak c_i}=kR_{N}
\)
and having Smith residue \(\boldsymbol\rho\).

It remains to identify the residue required by the class-matrix congruences.
For \(a_0=2k\), the congruences are
\[
P_N\mathbf v
\equiv
-k\boldsymbol\alpha_N
\pmod{D_N}.
\]
By definition of the residue vectors \(\boldsymbol\gamma_i\), and since \(\boldsymbol\gamma_m=\mathbf 0\), we have
\[
P_N\mathbf v
=
\sum_{i=0}^{m}\boldsymbol\gamma_i v_{\mathfrak c_i}.
\]
Therefore the required residue as an element of $\mathcal R_N$ is
\[
\sum_{i=0}^{m}\boldsymbol\gamma_i v_{\mathfrak c_i}
=
-k\boldsymbol\beta, \qquad \text{where}\ \boldsymbol\beta=\boldsymbol\alpha_N.
\]
Hence
\(
D_{m+1}(kR_{N},-k\boldsymbol\beta)
\)
counts the nonnegative cusp-order vectors satisfying the valence equation and the class congruences. By Theorem~\ref{mt}, these vectors are precisely the cusp-order vectors counted by \(L_N(k)\). Therefore
\(
|L_N(k)|
=
D_{m+1}(kR_{N},-k\boldsymbol\beta).
\)
\end{proof}

\begin{example}
We will apply the recursion in Theorem \ref{recur} to derive $|L_{11}(1)|$. Our construction of the  quasipolynomial for $N=11$ in Theorem \ref{level11} does not require this, but the small prime is a good case to illustrate the computational effectiveness of the recursion. We need to enumerate nonnegative integer vectors $(v_0, v_1, v_2, v_3, v_4, v_5)$ satisfying $$5v_{0}+v_{1}+v_{2}+v_{3}+v_{4}+v_{5}=5$$ and $$v_{0}+v_{1}+2v_{2}+4v_{3}+3v_{4}+2k\equiv0\pmod 5.$$ Write the coefficients of the congruence as $\gamma=(1,2,4,3,0).$ For $0\le j\le5$, let $D_{j}(S,r)$ denote the number of choices of $(v_{1},\ldots,v_{j})\in\mathbb{Z}_{\ge0}^{j}$ such that \begin{align} \label{sat1}
v_{1}+\cdots+v_{j}=S,\qquad\gamma_{1}v_{1}+\cdots+\gamma_{j}v_{j}\equiv r\pmod 5.
\end{align}
We initialize $D_{0}(0,0)=1$, and all other entries of $D_{0}$ are zero. This determines the initial table indexed by $j=0$ below. Then successive rows can be computed from prior rows. For example, $$D_{1}(1,1)=D_{0}(1,1)+D_{1}(0,0)=0+1,$$ and $$D_{1}(1,2)=D_{0}(1,2)+D_{1}(1-1,2-1)=D_{0}(1,2)+D_{1}(0,1)=0.$$
Proceeding with each row, we get the following arrays, where each ordered tuple records the tuple $(D_j(S,0),D_j(S,1),D_j(S,2),D_j(S,3),D_j(S,4)):$
\[
\begin{array}{c|c|c|c|c|c|c}

S & j=0 & j=1 & j=2 & j=3 & j=4 & j=5\\

\hline

0 &(1,0,0,0,0)&(1,0,0,0,0)&(1,0,0,0,0)&(1,0,0,0,0)&(1,0,0,0,0)&(1,0,0,0,0)\\

1 &(0,0,0,0,0)&(0,1,0,0,0)&(0,1,1,0,0)&(0,1,1,0,1)&(0,1,1,1,1)&(1,1,1,1,1)\\

2 &(0,0,0,0,0)&(0,0,1,0,0)&(0,0,1,1,1)&(1,1,1,2,1)&(2,2,2,2,2)&(3,3,3,3,3)\\

3 &(0,0,0,0,0)&(0,0,0,1,0)&(1,1,0,1,1)&(2,2,2,2,2)&(4,4,4,4,4)&(7,7,7,7,7)\\

4 &(0,0,0,0,0)&(0,0,0,0,1)&(1,1,1,1,1)&(3,3,3,3,3)&(7,7,7,7,7)&(14,14,14,14,14)\\

5 &(0,0,0,0,0)&(1,0,0,0,0)&(2,1,1,1,1)&(5,4,4,4,4)&(12,11,11,11,11)&(26,25,25,25,25)
\end{array}
\]
Thus, in the row labeled \(S\) and the column labeled \(j\), the ordered tuple
records how many choices of
\(
(v_1,\ldots,v_j)\in\mathbb Z_{\ge0}^{j}
\) satisfy \eqref{sat1}. Now fix $v_{0}=a$. 
Then the valence equation becomes
\(
S=v_1+v_2+v_3+v_4+v_5=5k-5a, \)
 and residues satisfy $v_{1}+2v_{2}+4v_{3}+3v_{4}\equiv-2k-a\pmod 5$. Thus the lattice count is assembled from the $j=5$ table by $$|L_{11}(k)|=\sum_{a=0}^{k}D_{5}(5k-5a,-2k-a).$$ For $k=1$, this gives $$|L_{11}(1)|=D_{5}(5,3)+D_{5}(0,2)=25+0=25.$$

The recursive method using the table above computes $30$ possible values, while the direct method would need to test $127$ tuples coming from the valence formula to determine if the congruence is satisfied. For fixed \(N\), the corresponding recursive table from Theorem \ref{recur} grows linearly in \(k\), while the
direct list of valence candidates requires \(O(k^m)\) steps. In the next section, this dynamic computation of lattice counts for small weights is used to interpolate quasipolynomial constituents and derive general formulas for $|L_N(k)|$. 
\end{example}

\section{Examples}

In this section, we illustrate computational strategies for enumerating the vectors in $L_{N}(k)$. Theorem~\ref{mt} implies there is a one-to-one correspondence between exponent vectors $\mathbf a \in L_{N}(k)$ and vectors of nonnegative integers  \(
\mathbf v \) satisfying the valence formula \eqref{valence} and
\begin{equation} \label{cong}
\bigl(P_N(\mathbf w(\mathbf v)-A_N\mathbf a^{(k)})\bigr)_r\equiv 0 \pmod{d_r},
\ 1\le r\le m,
\end{equation}
where we recall $m=\lfloor N/2\rfloor$, $P_{N}V_{N}Q_{N}=\mathrm{diag}(d_1,\dots,d_m)$ is the Smith normal form of the divisor class matrix $V_N$ and ${\bf a}^{(k)}$ is defined as in Lemma \ref{lem:psolution}.

\subsection{Genus 0}
For $N \ge 5$ such that the genus of $X_1(N)$ is zero, i.e., $N=5,\ldots,10,12$, the class matrix is trivial,  so \eqref{cong} is a tautology. In these cases $|L_N(k)|$ may be expressed in terms of simple binomial sums.
\begin{theorem}[Levels $N=5,\ldots,10,12$]
\label{genuszero}
For $N \ge 5$, when $X_1(N)$ has genus zero, the following enumerate $L_{N}(k)$:
\begin{center}
 \small
\setlength{\tabcolsep}{10pt}
 \renewcommand{\arraystretch}{1.15}
 \begin{tabular}{c l | c l}
\hline
$N$ & $\displaystyle |L_N(k)|$ & $N$ & $\displaystyle |L_N(k)|$ \\ \hline 

$5$  & $\displaystyle \sum_{j=0}^{\left\lfloor k/2\right\rfloor}
                 \binom{k-2j+1}{1}$
     & $8$  & $ \displaystyle \sum_{j=0}^{k}
                 \binom{2k-2j+3}{3}$ \\[12pt]

$6$  & $\displaystyle \binom{k+3}{3}$
     & $9$  & $\displaystyle \sum_{j=0}^{k}
                 \ \sum_{\ell=0}^{\left\lfloor \frac{3k}{2}-\frac{3j}{2}\right\rfloor}
                 \binom{3k-3j-2\ell+2}{2}$ \\[16pt]

$7$  & $\displaystyle \sum_{j=0}^{\left\lfloor 2k/3\right\rfloor}
                 \binom{2k-3j+2}{2}$
     & $10$ & $\displaystyle \sum_{j=0}^{\left\lfloor 3k/2\right\rfloor}
                 (j+1)\,\binom{3k-2j+3}{3}$ \\[16pt]

     \ &\ & $12$ & $\displaystyle \sum_{j=0}^{2k}
                 \binom{j+2}{2}\,\binom{4k-2j+3}{3}$  
                \\ \hline 
\end{tabular}
\end{center}
\end{theorem}

\begin{proof}
The case $N=12$ illustrates the combinatorial argument for the other cases. Choose cusp representatives for each Galois orbit $
\mathfrak c_0,\dots,\mathfrak c_6
$
with orbit sizes ordered so that
\[
(\Orb(\mathfrak c_0),\dots,\Orb(\mathfrak c_{6}))=(2,2,2,1,1,1,1).
\]
Then, with $v_{i} = \Ord_{\mathfrak c_i}(f_{\mathbf a}^{(N)})$, the valence formula \eqref{valence} becomes
\begin{equation} \label{val12}
2v_0+2v_1+2v_2+v_3+v_4+v_5+v_6=4k.
\end{equation}

Now group terms in \eqref{val12} with the same coefficients. The number of nonnegative triples \((v_0,v_1,v_2)\in \Bbb  Z^3\) with $
  v_0+v_1+v_2=j$
  is $\binom{j+2}{2}$, and the number of nonnegative \((v_3,v_4,v_5,v_6)\in \Bbb Z^4\) with
$ 
  v_3+v_4+v_5+v_6=4k-2j
$
  is
  $
  \binom{4k-2j+3}{3}$. Summing over $0 \le j \le 2k$ gives the claimed count.
\end{proof}
\begin{remark}
The formulas in Theorem \ref{genuszero} for $|L_N(k)|$ in the genus zero cases satisfy the general decomposition scheme \eqref{deft}--\eqref{defe} after formulas for sums of powers are applied to the binomial sums. The upper limits on the periods coming from Theorem \ref{genform} match the minimal periods. 
\end{remark}

\subsection{Higher genus}

For higher genera, when the class group is nontrivial, the same ideas allow us to compute the polynomial part $T_{N}(k)$ of the counting function but must be extended to include quasipolynomial parts to incorporate counts of cusp orders satisfying the valence formula \eqref{valence} and congruences \eqref{cong}. Fix $\omega_{n}=e^{2\pi i/n}$. We shall frequently make use of the fact  
\begin{align} \label{p_ind}
\frac{1}{n}\sum_{j=0}^{n-1}\omega_{n}^{j\ell }=\begin{cases}
1, & \ell\equiv0\pmod n,\\
0, & \ell \not\equiv0\pmod n.
\end{cases}
\end{align}
\begin{theorem}[Level $N=11$]\label{level11} Let $k\geq0$. Then
\[
|L_{11}(k)| = \frac{25 k^5}{24} + \frac{125 k^4}{24} + \frac{75 k^3}{8} + \frac{175 k^2}{24} + \frac{137 k}{60} + \frac{1}{5} + \mathcal{E}_{11}(k),
\]
where 
\[
\mathcal{E}_{11}(k)
=\begin{cases}
4/5, & k\equiv0\pmod 5,\\
-2/5, & k\equiv1\pmod 5,\\
2/5, & k\equiv2\pmod 5,\\
-4/5, & k\equiv3\pmod 5,\\
0, & k\equiv4\pmod 5.
\end{cases}
\]
\end{theorem}

\begin{proof}
We use the cusp-orbit representatives
\[
[0/1],\ [1/11],\ [2/11],\ [3/11],\ [4/11],\ [5/11].
\]
The orbit sizes are
\(
5,1,1,1,1,1.
\)
The valence formula is therefore, \begin{equation}\label{eq:N11-valence}
5v_{[0/1]}
+
v_{[1/11]}
+
v_{[2/11]}
+
v_{[3/11]}
+
v_{[4/11]}
+
v_{[5/11]}
=
5k
\end{equation}
Let
\[
\mathbf v
=
(v_{[0/1]},v_{[1/11]},v_{[2/11]},v_{[3/11]},v_{[4/11]})^T.
\]
By Lemma~\ref{prop:explicit-BN} and Lemma~\ref{lem:ANBN-equals-VN}, for the first five cusp representatives, 
\[
A_{11}
=
\frac1{132}
\begin{pmatrix}
11&11&11&11&11\\
61&13&-23&-47&-59\\
13&-47&-59&-23&61\\
-23&-59&13&61&-47\\
-47&-23&61&-59&13
\end{pmatrix}, \quad 
B_{11}
=
\begin{pmatrix}
3&-2&4&0&-5\\
2&-3&1&-2&-3\\
-2&1&-4&1&0\\
-3&2&-1&-2&-2\\
0&2&0&3&-2
\end{pmatrix}.
\]
A Smith normal form of \(V_{11} = A_{11}B_{11}\) is
\[
P_{11}V_{11}Q_{11}=\operatorname{diag}(1,1,1,1,5),
\]
where
\[
P_{11}=
\begin{pmatrix}
0&0&1&0&0\\
0&-1&0&-1&0\\
0&2&1&2&1\\
0&-13&-6&-12&-9\\
-1&-26&-12&-24&-18
\end{pmatrix}
, \quad 
Q_{11}=
\begin{pmatrix}
1&-1&-3&-28&70\\
0&1&1&-22&56\\
0&0&1&10&-25\\
0&0&0&15&-38\\
0&0&0&-2&5
\end{pmatrix}.
\]
The corresponding \(g\)-vector is
\[
\mathbf g=(1,11,11,11,11)^T.
\]
Hence
\[
\mathbf w(\mathbf v)
=
\mathbf v-\frac{k}{12}\mathbf g
=
\mathbf v-\frac{k}{12}\mathbf g.
\]
By Lemma~\ref{lem:psolution} and Theorem \ref{mt}, we require 
\[
P_{11}\left(\mathbf w(\mathbf v)-A_{11}\mathbf a^{(k)}\right)
\in
\operatorname{diag}(1,1,1,1,5)\mathbb Z^5,
\]
where
\[
\mathbf a^{(k)}=(-3k,3k,-k,0,0)^T.
\]
The nontrivial condition involves the fifth row of $P_{11}$: $\mathbf p_{5}=(-1,-26,-12,-24,-18)$, 
\[
\mathbf p_{5}\left(\mathbf w(\mathbf v)-A_{11}\mathbf a^{(k)}\right)
=
\mathbf p_{5}\mathbf v
+ 18k.
\]
Modulo \(5\), this becomes
\[
4v_{[0/1]}+4v_{[1/11]}+3v_{[2/11]}
+v_{[3/11]}+2v_{[4/11]}+3k
\equiv0\pmod5.
\]
Therefore, 
\[
L_{11}(k)
=
\left\{
A_{11}^{-1}\,\mathbf{w}(\mathbf{v})
\;\middle|\;
\begin{aligned}
&v_{[0/1]},\, v_{[1/11]},\, v_{[2/11]},\, v_{[3/11]},\, v_{[4/11]},\, v_{[5/11]}
\in \mathbb{Z}_{\ge 0}, \\[1pt]
&5v_{[0/1]} + v_{[1/11]} + v_{[2/11]} + v_{[3/11]} + v_{[4/11]} + v_{[5/11]}
= \frac{5a_0}{2}, \\[3pt]
&v_{[0/1]} + v_{[1/11]} + 2v_{[2/11]} + 4v_{[3/11]} + 3v_{[4/11]} + a_0
\equiv 0 \pmod{5}
\end{aligned}
\right\}.
\]
By re-writing \eqref{eq:N11-valence}, the lattice points in the polytope can be counted by
 \[
|L_{11}(k)|=\sum_{0\le t_0 \le \frac{a_{0}}{2}}\sum_{{t_{i}\ge0,\; 1 \le i \le 5 \atop \sum t_{i} =  \frac{5a_{0}}{2}-5t_0}}\frac{1}{5}\sum_{j=0}^{4}\omega_{5}^{j\left(t_0 + t_1 + 2t_2+ 4t_3 + 3t_4 + a_0\right)}.
\]
Interchange the sums to get
\begin{align} \label{11p}
|L_{11}(k)|=\frac{1}{5}\sum_{j=0}^{4}\omega_{5}^{ja_{0}}\sum_{0\le t_0\le\frac{a_{0}}{2}}\omega_{5}^{jt_0}\sum_{{t_{i}\ge0\atop \sum t_{i} =5\frac{a_{0}}{2}-5t_0}}\omega_{5}^{j\left(t_1 + 2t_2+ 4t_3 + 3t_4 \right)}.
\end{align}
Consider the generating function for  $W_{j}(n)$, the number of
tuples  \begin{align} \label{tup11}
(t_{1},t_{2},t_{3},t_{4},t_5) \quad \text{such that}\quad t_{1}+t_{2}+t_{3}+t_{4}+t_5=n, \qquad t_i \ge 0,
\end{align}
weighted by $\omega_{5}^{j\left(t_1 + 2t_2+ 4t_3 + 3t_4\right)}$. For $1 \le j \le 4$, this reduces to
\begin{align*}
 & \sum_{n=0}^{\infty}W_{j}(n)x^{n}=\sum_{n=0}^{\infty}\sum_{{t_{i}\ge0, 1 \le i \le 5 \atop \sum t_{i}=n}}\omega_{5}^{j\left(t_1 + 2t_2+ 4t_3 + 3t_4\right)}x^{n}
 = \frac{1}{1-x}\prod_{k=1}^{4} (1 - \omega_{5}^{kj}x)^{-1} 
 = \frac{1}{1-x^{5}}.
\end{align*}
It follows that, for $1 \le j \le 4$,
\[
W_j(n)=
\begin{cases}
1, & n\equiv 0\pmod 5,\\
0, & \text{otherwise}.
\end{cases}
\]
Note that the innermost sum of \eqref{11p} can be written in terms
of $W_j$ and evaluated via combinations with replacement through \eqref{tup11} when $j=0$:
\begin{align*}
\sum_{{t_{i}\ge0,\; 1 \le i \le 5 \atop \sum t_{i} =  \frac{5a_{0}}{2}-5t_0}}\omega_{5}^{j\left(t_1 + 2t_2+ 4t_3 + 3t_4 \right)} = W_j \left ( 5 \left (\frac{a_{0}}{2}-t_0 \right ) \right )= \begin{cases}
{5 \left (\frac{a_{0}}{2}-t_0 \right ) +4 \choose 4}, & j=0, \\ 1, & 1 \le j \le 4.
\end{cases}
\end{align*}
Therefore, separating the terms $j=0$ from $1 \le j \le 4$ and manipulating the result, we obtain
\begin{align*}
|L_{11}(k)| &=\frac{1}{5}\sum_{t_0=0}^{\frac{a_{0}}{2}} {5 \left (\frac{a_{0}}{2}-t_0 \right ) + 4 \choose 4 } + \frac{1}{5}\sum_{j=1}^{4}\omega_{5}^{ja_{0}}\sum_{t_0=0}^{\frac{a_{0}}{2}}\omega_{5}^{jt_0} \\ &= \frac{1}{5}\sum_{t_0=0}^{\frac{a_{0}}{2}} \binom{5t_0+4}{4} +  \frac{1}{5}\sum_{j=1}^{4}\omega_{5}^{ja_{0}/2}\frac{1-\omega_{5}^{3j(1+a_{0}/2)}}{1-\omega_{5}^{3j}}.
\end{align*}
Expanding the binomial coefficient and applying formulas for sums of powers gives the claimed main term. Since $\omega_{5}=e^{2\pi i/5}$, the final sum above depends only on the
residue class for $a_{0}/2$ modulo 5. Simplifying in each case, we
get the claimed quasipolynomial part $\mathcal{E}_{11}$.
\end{proof}

By Theorem \ref{genform}, a period of the quasipolynomial for $|L_{13}(k)|$ divides $114$. By simplifying the sums of roots of unity obtained from \eqref{thel} and applying Theorem \ref{recur}, proper divisors are ruled out.
\begin{theorem}[Level $N=13$] \label{lev:13}
For $k\geq0$, we have $|L_{13}(k)| = T_{13}(k) + \mathcal{E}_{13}(k)$, where 
\begin{align*}
T_{13}(k)
=
\frac{117649}{82080}k^6
+\frac{16807}{2280}k^5
+\frac{156065}{10944}k^4
+\frac{5831}{456}k^3
+\frac{71981}{13680}k^2
+\frac{959}{1140}k
+\frac{27749}{196992},
\end{align*}
and 
\[
\mathcal E_{13}(k)
=
\frac{10368\,C_{13}(k)+d_{k\bmod 6}}{196992},
\]
where
\[
(d_0,d_1,d_2,d_3,d_4,d_5)
=
(-17381,-18971,-22373,-24347,-22373,-18971).
\]
    and where $C_{13}(19q +r)$ is periodic modulo $114$ given by Table \ref{tab:E13-error}.
\begin{table}[ht] 
  \centering
  \caption{Periodic term values:
           the entry in row $q$ and column $r$ is $C_{13}(19q+r)$.} 
  \renewcommand{\arraystretch}{1}
  \resizebox{\textwidth}{!}{%
    \begin{tabular}{|c||ccccccccccccccccccc|} \hline  
      $q \backslash r$
        & 0 & 1 & 2 & 3 & 4 & 5 & 6 & 7 & 8 & 9
        & 10 & 11 & 12 & 13 & 14 & 15 & 16 & 17 & 18  \\
      \hline\hline
      0
        & 18 & 0  & 60  & 30  & 0
        & -54 & -54 & -60 & 60 & 48
        & 42 & -6 & 18  & -60 & -12
        & 0  & 0  & -24 & 18 \\ 
      1
        & 0  & 60 & 54  & 36 & -60
        & -36 & -72 & 0  & 54 & 54
        & -18 & 12 & 0  & 0  & -18
        & 6  & -60 & -6 & 0 \\  
      2
        & 60 & 54 & 60  & -24 & -42
        & -54 & -12 & -6 & 60 & -6
        & 0  & -6 & 60 & -6 & -12
        & -54 & -42 & -24 & 60 \\ 
      3
        & 54 & 60 & 0  & -6 & -60
        & 6  & -18 & 0  & 0  & 12
        & -18 & 54 & 54 & 0  & -72
        & -36 & -60 & 36 & 54 \\ 
      4
        & 60 & 0  & 18 & -24 & 0
        & 0  & -12 & -60 & 18 & -6
        & 42 & 48 & 60 & -60 & -54
        & -54 & 0  & 30 & 60 \\ 
      5
        & 0  & 18 & 0  & 36 & -6
        & 6  & -72 & -42 & 0  & 54
        & 36 & 54 & 0  & -42 & -72
        & 6  & -6 & 36 & 0 \\ \hline 
    \end{tabular}
  }\label{tab:E13-error}
\end{table}
\end{theorem}

\begin{proof}
For $N=13$, the valence formula takes the form
\begin{align}
6v_{[0/1]}+v_{[1/13]}+v_{[2/13]}+v_{[3/13]}+v_{[4/13]}+v_{[5/13]}+v_{[6/13]}=7k.
\label{eq:13-val}
\end{align}
The corresponding congruence simplifies to
\begin{align}
v_{[0/1]}+14v_{[1/13]}+11v_{[2/13]}+2v_{[3/13]}+13v_{[4/13]}+18v_{[5/13]}+2k\equiv 0\pmod{19}.
\label{eq:13-cong}
\end{align}
By proceeding as in the case of $N=11$ and expanding \eqref{thel} in terms of roots of unity, one obtains from the trivial index a binomial sum. In particular, 
\begin{align} \label{l13}
|L_{13}(k)|
= \frac{1}{19}\sum_{t_0=0}^{\lfloor 7k/6\rfloor}\binom{7k-6t_0+5}{5} +\frac{1}{19}\sum_{j=1}^{18}\omega_{19}^{2jk}
\sum_{t_0=0}^{\lfloor 7k/6\rfloor}\omega_{19}^{jt_0}W_j(7k-6t_0),
\end{align}
where
\begin{align*}
W_j(n):=
\sum_{\substack{t_1,\dots,t_6\ge 0\\ t_1+\cdots+t_6=n}}
\omega_{19}^{j(14t_1+11t_2+2t_3+13t_4+18t_5)}, \qquad 1\le j \le 18. 
\end{align*}
The generating function is a product of geometric series with a partial fraction,
\begin{align} \label{w13}
\sum_{n\ge 0}W_j(n)x^n
 = \prod_{v \in S}\frac{1}{1-\omega_{19}^{vj}x}
 =
 \sum_{v\in S}\frac{A_{j,v}}{1-\omega_{19}^{jv}x}, \quad A_{j,v}\in \mathbb C, \ S:=\{0,14,11,2,13,18\}.
\end{align}
Expanding both sides of \eqref{w13} and equating coefficients, the summands of \eqref{l13} can be written
\begin{align}
\omega_{19}^{2jk}\sum_{t_0=0}^{\lfloor 7k/6\rfloor}\omega_{19}^{jt_0}W_j(7k-6t_0)  
 \nonumber 
 &=
\sum_{v\in S}
A_{j,v}\,
\omega_{19}^{j(2+7v)k}
\sum_{t_0=0}^{\lfloor 7k/6\rfloor}\omega_{19}^{j(1-6v)t_0}\\ 
&=\sum_{v\in S}
A_{j,v}\,
\omega_{19}^{j(2+7v)k}
\frac{1-\omega_{19}^{j(1-6v)(\lfloor 7k/6\rfloor+1)}}{1-\omega_{19}^{j(1-6v)}}.
\label{eq:13-jpiece}
\end{align}
Now write
\(
k=114m+r,\ 
0\le r\le 113.
\)
Since \(114\equiv 0\pmod{19}\),
\(
\omega_{19}^{j(2+7v)k}=\omega_{19}^{j(2+7v)r},
\)
and note
\[
\left\lfloor \frac{7k}{6}\right\rfloor
=
\left\lfloor \frac{7(114m+r)}{6}\right\rfloor
=
133m+\left\lfloor \frac{7r}{6}\right\rfloor,
\]
Hence
\(
\omega_{19}^{j(1-6v)(\lfloor 7k/6\rfloor+1)}
=
\omega_{19}^{j(1-6v)(\lfloor 7r/6\rfloor+1)},
\)
so terms in \eqref{eq:13-jpiece} depend only on \(r\equiv k\bmod 114\). Therefore, by \eqref{l13}, for some sequence $C_{13}(k)$ periodic modulo $114$,
\begin{align} \label{13f}
    |L_{13}(k)| = \frac{1}{19}\sum_{t_0=0}^{\lfloor 7k/6\rfloor}\binom{7k-6t_0+5}{5} +\frac{1}{19}C_{13}(k).
\end{align}
The constants $C_{13}(k)$ can be computed by evaluating $|L_{13}(k)|$, $0 \le k \le 113$ through Theorem \ref{recur}. Finally, $T_{13}(k)$ and $\mathcal{E}_{13}(k)$ may be computed from \eqref{13f} and by applying \eqref{deft} and \eqref{defe}.
\end{proof}
 
For $N=11,13$, the nontrivial indices in \eqref{thel} contribute a periodic remaining term, so that $\mathcal{E}_{N}$ is a periodic sequence. In general, the nontrivial indices in \eqref{thel} correspond to $\mathcal{E}_{N}(k)$ of degree greater than $0$. Knowledge of the degree of the quasipolynomial for $|L_{N}(k)|$ from Theorem \ref{mt} allows us to compute the generating functions for each quasipolynomial constituent through Theorem \ref{recur} and then apply polynomial interpolation. This is illustrated in the proof of Theorem \ref{level14}.

\begin{theorem}[Levels $N=14, 15$]
\label{level14}
Let $k\geq0$, and
\begin{align*}
T_{14}(k)
&=
\frac{(k+1)\left(324k^6+1944k^5+4482k^4+4968k^3+2788k^2+824k+105\right)}{315},
\\ T_{15}(k)
&=
\frac{(k+1)(2k+1)(2k+3)
\left(512k^4+2048k^3+2584k^2+1072k+105\right)}
{1260}.
\end{align*}
Then 
\[
|L_{N}(k)|
=
T_{N}(k)
+
\mathcal E_{N}(k), \qquad N=14, 15,
\]
where
\[
\mathcal E_{14}(k)=
\begin{cases}
\dfrac{2(k+3)}{9},
& k\equiv0,3\pmod6,\\[2mm]
\dfrac{2(k-1)}{9},
& k\equiv1,4\pmod6,\\[2mm]
-\dfrac{4(k+1)}{9},
& k\equiv2,5\pmod6.
\end{cases}
\qquad 
 \mathcal E_{15}(k)=
\begin{cases}
\dfrac{k+3}{4},
& k\equiv0\pmod4,\\[3mm]
-\dfrac{k+1}{4},
& k\equiv1,3\pmod4,\\[3mm]
\dfrac{k-1}{4},
& k\equiv2\pmod4.
\end{cases}
\]
\end{theorem}
\begin{proof}
We prove the claims for $N=14$, since $N=15$ is analogous. From the class group conditions, 
\begin{equation}\label{eq:N14-corrected-cong}
v_0+4v_1+5v_2+v_3+3v_4+4v_5+2v_6
\equiv4k\pmod6.
\end{equation}
Now write
\(
k=6t+r,
\) $0\le r\le5.$
The valence condition takes the form
\[
3v_0+3v_1+v_2+v_3+v_4+v_5+v_6+v_7=6k=36t+6r.
\]
Then, with $b_{r,t}=|L_{14}(6t+r)|,$
we obtain the first eight values for each residue class modulo $6$:
\[
\resizebox{\textwidth}{!}{$
\begin{array}{c|rrrrrrrr}
r & b_{r,0}&b_{r,1}&b_{r,2}&b_{r,3}&b_{r,4}&b_{r,5}&b_{r,6}&b_{r,7}\\
\hline
0&1&827119&64097149&916443883&6266192761&28264352359&97561035253&279408496483\\
1&98&2118068&107781638&1312736408&8247043754&35301200252&117589551182&328202186336\\
2&1973&4849215&174831913&1847653043&10742013981&43789172647&141043660433&384124346139\\
3&15664&10166338&274849300&2559447526&13857801976&53969519242&168398302684&448022495278\\
4&76685&19851515&420361865&3494390135&17718131717&66114562931&200179697153&520822917359\\
5&278734&36558044&627448170&4707966904&22465713926&80530609812&236969387938&603536027504
\end{array}
$}
\]
These values can be computed from Theorem \ref{recur} or by extracting coefficients $b_{r,t} = [x^{36t+6r}]\Phi_r(x)$ from the corresponding generating function, defined
for \(0\le r\le5\) by
\begin{align*}
\Phi_r(x)
=
\frac16\sum_{j=0}^{5}
\frac{\omega_6^{-4rj}}
{(1-\omega_6^{j}x^3)
(1-\omega_6^{4j}x^3)
(1-\omega_6^{5j}x)
(1-\omega_6^{j}x)
(1-\omega_6^{3j}x)
(1-\omega_6^{4j}x)
(1-\omega_6^{2j}x)
(1-x)}.
\end{align*}
By Theorem \ref{int}, the polynomial constituents are given, for $k \equiv r \pmod{6}$, by $|L_{N}(k)| = Q_{r}(k)$,
where $$Q_{r}(k)=\sum_{\ell=0}^{7}
p_{r,\ell}
\binom{\frac{k-r}{6}+7-\ell}{7}, \qquad p_{r,s} = \sum_{a=0}^{s}
(-1)^a\binom{8}{a}b_{r,s-a}.$$ Expressions for $T_{14}(k)$ and $\mathcal{E}_{14}(k)$ may be obtained from these by applying \eqref{deft} and \eqref{defe}.
\end{proof}

When the class group is not cyclic, the situation is more subtle, and a multivariate generating function results. This is illustrated in the following example for $N=16$ whose class group has invariant factors $2, 10$. In this case the polytope enumerating function can be derived by first forming the rational generating function.

\begin{theorem}[Level \(N=16\)] \label{l16}
For $k \ge0$, 
\(
|L_{16}(k)|
= T_{16}(k) + \mathcal{E}_{16}(k),
\)
where
\[T_{16}(k)=\frac{2048k^8}{1575}
+\frac{2048k^7}{225}
+\frac{1984k^6}{75}
+\frac{9296k^5}{225}
+\frac{958k^4}{25}
+\frac{5138k^3}{225}
+\frac{33851k^2}{3150}
+\frac{137k}{30}
+\frac{109}{125}\]
and
\[
\mathcal E_{16}(k)=
\begin{cases}
-\dfrac{64}{125}, & k\equiv1\pmod5,\\[2mm]
\dfrac{16}{125}, & k\equiv0,2,3,4\pmod5.
\end{cases}
\]
\end{theorem}

\begin{proof}
For $N=16$, we have the conditions
\begin{equation}\label{eq:N16-valence-corrected}
4t_0+2t_1+2t_2+t_3+t_4+t_5+t_6+t_7+t_8=8k.
\end{equation}
and the two congruences
\begin{equation}\label{eq:N16-cong1-corrected}
t_1+t_2+t_3+t_4\equiv 0\pmod 2,
\end{equation}
\begin{equation}\label{eq:N16-cong2-corrected}
8t_0+9t_1+4t_2+2t_3+7t_4+3t_5+t_6+4t_7+4k
\equiv0\pmod {10}.
\end{equation}
Then by \eqref{thel}
\begin{align}
|L_{16}(k)|
=
\frac1{20}\sum_{u=0}^{1}\sum_{j=0}^{9}
\omega_{10}^{4jk}W_{u,j}(8k),
\label{eq:N16-filtered-corrected}
\end{align}
where
\[
W_{u,j}(n)
=
\sum_{\substack{t_0,\dots,t_8\ge 0\\
4t_0+2t_1+2t_2+t_3+t_4+t_5+t_6+t_7+t_8=n}}
(-1)^{u(t_1+t_2+t_3+t_4)}
\omega_{10}^{j(8t_0+9t_1+4t_2+2t_3+7t_4+3t_5+t_6+4t_7)}.
\label{eq:N16-Wuj-corrected}
\]
Apply \eqref{eq:N16-filtered-corrected} to get 
\[
F_{16}(q):=\sum_{k\ge0}|L_{16}(k)|q^k = \frac1{20}\sum_{u=0}^{1}\sum_{j=0}^{9}H_{u,j}(q).
\]
where, for each pair \((u,j)\), we define
\[
H_{u,j}(q)=\sum_{k\ge0}\omega_{10}^{4jk}W_{u,j}(8k)q^k.
\]
The functions \(H_{u,j}(q)\) may be obtained by extracting terms whose degree in \(x\) is divisible by \(8\) from 
\begin{align*}
G_{u,j}(x) &=\sum_{n\ge0}W_{u,j}(n)x^n =
\frac{1}{
(1-\omega_{10}^{8j}x^4)
(1-(-1)^u\omega_{10}^{9j}x^2)
(1-(-1)^u\omega_{10}^{4j}x^2)} \\  &
\hspace{1cm}\times
\frac{1}{
(1-(-1)^u\omega_{10}^{2j}x)
(1-(-1)^u\omega_{10}^{7j}x)
(1-\omega_{10}^{3j}x)
(1-\omega_{10}^{j}x)
(1-\omega_{10}^{4j}x)
(1-x)
}.
\end{align*}

A computer algebra system can be used to simplify the rational functions resulting from each of the $20$ indices $(u,j)$ to obtain
\begin{align} \label{eq:N16-claimedGF}
F_{16}(q)
=
\frac{
P(q)
}
{{(1-q)^9(1+q+q^2+q^3+q^4)}},
\end{align}
where the second factor of the denominator is the cyclotomic polynomial of index $5$, and
\[
\begin{aligned}
P(q) = &4q^{11} + 409q^{10} + 7219q^9 + 29434q^8 + 48266q^7 
 + 52313q^6 + 51974q^5 \\ &\quad+ 45245q^4 + 22982q^3  + 4150q^2 + 147q + 1.
\end{aligned}
\]
The claimed expressions for $T_{16}(k)$ and $\mathcal{E}_{16}(k)$ may be derived from this expansion. 
\end{proof}

\end{document}